\documentclass[11pt]{article}
\usepackage{latexsym}
\usepackage{amsmath}
\usepackage{amsthm}
\usepackage{amssymb}
\usepackage{esint}

\newtheorem{theorem}{Theorem}[section]
\newtheorem{corollary}[theorem]{Corollary}

\newtheorem{example}[theorem]{Example}
\newtheorem{lemma}[theorem]{Lemma}
\newtheorem{proposition}[theorem]{Proposition}
\newtheorem{conjecture}[theorem]{Conjecture}

\newcommand{\vanish}[1]{}

\newcommand{\pair}[2]{\left(#1,#2\right)}
\newcommand{\erf}{\operatorname{erf}}
\newcommand{\wt}{\operatorname{wt}}

\numberwithin{equation}{section}

\begin{document}

\title{A Spectral Approach to Consecutive Pattern-Avoiding Permutations}

\author{Richard Ehrenborg,
        Sergey Kitaev and
        Peter Perry}
\date{}

\maketitle

\begin{abstract}
We consider the problem of enumerating permutations in the symmetric
group on $n$ elements which avoid a given set of consecutive patterns
$S$, and in particular computing asymptotics as $n$ tends to
infinity. We develop a general method which solves this enumeration
problem using the spectral theory of integral operators on
$L^{2}([0,1]^{m})$, where the patterns in $S$ have length $m+1$.
Kre\u{\i}n and Rutman's generalization of the Perron--Frobenius
theory of non-negative matrices plays a central role. Our methods
give detailed asymptotic expansions and allow for explicit
computation of leading terms in many cases.
As a corollary to our results,
we settle a conjecture of Warlimont on asymptotics for the number of
permutations avoiding a consecutive pattern.
\end{abstract}

\section{Introduction}

In this paper, we study integral operators of the form
\begin{equation}
(T f)(x_{1},\ldots,x_{m})
    =
\int_{0}^{1} \chi(t,x_{1},\ldots,x_{m})
    \cdot
f(t,x_{1},\ldots,x_{m-1}) dt
\label{equation_T}
\end{equation}
and their application to enumerating permutations that avoid a consecutive
pattern. Here $\chi$ is a real-valued function on $\left[0,1\right]^{m}$
which takes values in $\left[0,1 \right]$ and is continuous away from a set
of measure zero in $\left[0,1 \right]^{m+1}$. As we will show, operators of
this type arise naturally when counting permutations that avoid a
consecutive pattern of length $m+1$.

To define the enumeration problem, let $\mathfrak{S}_{n}$ denote the
symmetric group on $n$ elements. For $\pi\in\mathfrak{S}_{n}$ we write
$\pi=\left( \pi_{1}\pi_{2}\cdots\pi_{n}\right)$ where the $\pi_{k}$ are the
integers from $1$ to $n$. For $x=(x_{1},\ldots,x_{n})\in\mathbb{R}^{n}$ with
$x_{i}\neq x_{j}$ for all distinct $i$ and $j$, we denote by $\Pi(x)$ the
unique permutation $\pi\in\mathfrak{S}_{n}$ with $\pi_{i}<\pi_{j}$ if and
only if $x_{i}<x_{j}$. A {\em pattern} of length $k$ is an element $\sigma$
of $\mathfrak{S}_{k}$. If $\pi$ is a permutation of length $n\geq k$, $\pi$
avoids $\sigma$ if $\Pi(\pi_{j},\pi_{j+1},\ldots,\pi_{j+k-1})\neq\sigma$ for
all $j$ with $1\leq j\leq n-k+1$. More generally, if
$S\subseteq\mathfrak{S}_{k}$
we say that $\pi$ avoids $S$ if $\pi$ avoids each
$\sigma \in S$. That is, $S$ is the set of forbidden patterns.

For a subset $S$ of $\mathfrak{S}_{m+1}$, denote by $\alpha_{n}(S)$ the number
of permutations $\pi\in\mathfrak{S}_{n}$ that avoid $S$. Observe that for
$n \leq m$ we have $\alpha_{n}(S)=n!$. Our goal is to compute asymptotics of
$\alpha_{n}(S)$ as $n$ tends to infinity.
Throughout the paper we will assume that $m \geq 2$.

For $S\subseteq\mathfrak{S}_{m+1}$ we define a function $\chi$ on
$\left[0,1\right]^{m+1}$ by setting
$\chi(x)=0$ if $x_{i}=x_{j}$ for some
distinct indices $i$ and $j$, and otherwise setting
\begin{equation}
\chi(x)
  =
\left\{ \begin{array}{c l}
           1 & \Pi(x)\notin S, \\
           0 & \Pi(x)\in S.
        \end{array} \right.
\label{equation_chi}
\end{equation}
Finally, let $T$ be the operator of the form~(\ref{equation_T}).
The standard inner product
on $L^{2}\left( \left[0,1\right]^{m}\right)$,
is defined by
$$      \pair{f}{g}
     =
        \int_{[0,1]^{m}} f(x) \cdot \overline{g(x)} dx ,
$$
and hence the $L^{2}$ norm is given
by 
$\| f \| = \sqrt{\pair{f}{f}}$.
Moreover, let $\mathbf{1}$ denote the
constant function $1$ on~$\left[0,1\right]^{m}$.
It is straightforward to prove
(see Proposition~\ref{proposition_key}) that the formula
\begin{equation}
\frac{\alpha_{n}(S)}{n!}=\pair{T^{n-m}(\mathbf{1})}{\mathbf{1}}
\label{equation_key}
\end{equation}
holds for $n \geq m$.
Note that the left-hand side is the probability of selecting
a permutation $\pi\in\mathfrak{S}_{n}$
at random that avoids the set $S$.

The asymptotic behavior of powers of a bounded linear operator is determined
by its spectrum. Recall that if $A$ is a bounded linear operator from a
Hilbert space to itself, the resolvent set of $A$ is the set of all
$z\in\mathbb{C}$
so that $(z I - A)^{-1}$ is also a bounded operator. The complement
in $\mathbb{C}$ of the resolvent set is the spectrum of $A$, denoted
$\sigma(A)$, and the spectral radius of $A$ is given by
$$
r(A)
       =
\sup\left\{ \left| \lambda\right| :\lambda\in\sigma(A)\right\} .
$$
The peripheral spectrum of $A$ is the intersection of $\sigma(A)$ and the
circle of radius $r(A)$ in the complex plane. As we will see, the peripheral
spectrum of the operator $T$ consists at least of a real eigenvalue at
$r(A)$, although this need not be the only eigenvalue in the peripheral
spectrum. Finally,
define the adjoint of an operator $A$ to be
the operator $A^{*}$ that satisfies
$\pair{f}{A^{*}(g)} = \pair{A(f)}{g}$.

Using spectral theory, we obtain:
\begin{theorem}
Let $S$ be a set of forbidden patterns in $\mathfrak{S}_{m+1}$.
Then the nonzero spectrum of the associated operator $T$
consists of discrete eigenvalues of
finite multiplicity which may accumulate only at $0$.
Furthermore, let $r$ be a positive real number 
such that there is no eigenvalue of $T$ with modulus~$r$
and let
$\lambda_{1}, \ldots, \lambda_{k}$
be the eigenvalues of $T$ greater in modulus than $r$.
Assume that $\lambda_{1}, \ldots, \lambda_{k}$
are simple eigenvalues, with associated eigenfunctions $\varphi_{i}$
and that the adjoint operator $T^{*}$ has
eigenfunctions $\psi_{i}$ corresponding the eigenvalues
$\lambda_{i}$.
Then we have the expansion
\begin{equation}
    \alpha_{n}(S)/n!
    =
\pair{T^{n-m}(\mathbf{1})}{\mathbf{1}}
    =
\sum_{i=1}^{k} 
     \frac{\pair{\varphi_{i}}{\mathbf{1}} \cdot
           \pair{\mathbf{1}}{\overline{\psi_{i}}}}
          {\pair{\varphi_{i}}{\overline{\psi_{i}}}} 
     \cdot \lambda_{i}^{n-m}
     +
O(r^{n})   .
\label{equation_expansion_I}
\end{equation}
Moreover, when the operator $T$ has a positive spectral
radius, that is, $r(T) > 0$ then the spectral radius is
an eigenvalue of the operator $T$.
\label{theorem_expansion}
\end{theorem}
Observe that when $T$ has spectral radius $0$,
this result is not useful.

In applications the error term in
equation~(\ref{equation_expansion_I})
can be selected to be the modulus of the next eigenvalue,
that is,
\begin{equation}
    \alpha_{n}(S)/n!
    =
\sum_{i=1}^{k} c_{i} \cdot \lambda_{i}^{n}
     +
O(|\lambda_{k+1}|^{n})   ,
\label{equation_expansion_II}
\end{equation}
where the eigenvalues satisfy
$|\lambda_{1}| \geq \cdots \geq |\lambda_{k}| > |\lambda_{k+1}|$
and
$\lambda_{k+1}$ is a simple eigenvalue.
This follows from letting $r$ be smaller than $|\lambda_{k+1}|$
given a few more terms 
in equation~(\ref{equation_expansion_I}) and increasing the error term.

For certain sets of patterns, we can show that there is a unique
largest eigenvalue with respect to modulus and that this eigenvalue is
simple and positive.  Thus we obtain the leading term of the
asymptotics of $\alpha_{n}(S)/n!$.
A sufficient
condition for the peripheral spectrum of $T$ to consist only of the
simple eigenvalue $r(T)$ is as follows. If $(X,\mu)$ is
a measure space and $f\in L^{2}(X,\mu)$ is a real-valued function, we will
say that $f>0$ if $f(x)>0$ for almost every $x\in X$, and $f\geq 0$ if
$f(x)\geq 0$ for almost every $x$. A bounded operator $A$ on $L^{2}(X,\mu)$
is {\em positivity improving} if for any $f\geq 0$ different from $0$ there
is an integer $k$ (possibly depending on $f$) so that $A^{k}f>0$.
Kre\u{\i}n and Rutman~\cite[Theorem~6.3]{Krein_Rutman}
showed that,
for such an operator, 
the spectral radius $r(A)$ is 
a simple eigenvalue $r(A)$ and all other eigenvalues
$\lambda$ are smaller than $r(A)$, that is,
$|\lambda| < r(A)$.
Furthermore, the associated
eigenfunction~$\varphi$ is positive almost everywhere.
Finally, the adjoint operator $A^{*}$ also has a largest real, positive
and simple eigenvalue at $r(A)$ and
an almost everywhere positive eigenfunction $\psi$.
We summarize this discussion in:
\begin{theorem}
If the operator $A$ is positivity improving
then its largest eigenvalue is real, positive and simple.
\label{theorem_positivity_improving}
\end{theorem}

A sufficient condition for $T$ to be positivity improving can be
formulated in combinatorial terms as follows.
If $S$ is a set of patterns,
let $G_{S}$
be the graph with vertex set $\mathfrak{S}_{m}$
and a directed edge from
$\pi$ to $\sigma$ if there is a permutation
$\tau\in\mathfrak{S}_{m+1}\backslash S$
with $\Pi(\tau_{1},\ldots,\tau_{m})=\pi$ and
$\Pi(\tau_{2},\ldots,\tau_{m+1}) = \sigma$.
The graph $G_{\emptyset}$ is known
as the graph of overlapping permutations.
Recall that a graph $G$ is {\em strongly connected}
if any vertex of $G$ is connected to any
other vertex by a directed path.

\begin{theorem}
Suppose that the graph $G_{S}$ is strongly connected and that the monotone
permutations $12\cdots (m+1)$ and $(m+1) \cdots 21$ do not belong to $S$. Then
the operator $T$ is positivity improving.
\label{theorem_G}
\end{theorem}

As a corollary we have the following result,
proving a conjecture of Warlimont~\cite{Warlimont}.
See also Theorem~4.1 in~\cite{Elizalde}.
\begin{corollary}
Let $S$ consist of a single permutation $\sigma$.
Then the asymptotics of
$\alpha_{n}(S)$ 
is given by
$\alpha_{n}(S)/n! = c \cdot \lambda^{n} + O(r^{n})$,
where $c$, $\lambda$ and $r$ are positive constants such that $\lambda > r$.
\label{corollary_single_permutation}
\end{corollary}
\begin{proof}
When the permutation $\sigma$
differs from the two monotone permutations
$12\cdots (m+1)$ and $(m+1) \cdots 21$,
Theorem~\ref{theorem_G} applies to the forbidden set $S = \{\sigma\}$.
The two remaining cases are equivalent, and 
can be settled by Theorem~\ref{theorem_graph_H_ergodic}.
\end{proof}

Another application of Theorem~\ref{theorem_G} is the following.
Call a permutation $\pi$ in $\mathfrak{S}_{n}$
{\em decomposable} if there exists an index $i$
such that
$1 \leq i \leq n-1$ and
$\pi(1), \ldots, \pi(i) \leq i$
(equivalently
$i+1 \leq \pi(i+1), \ldots, \pi(n)$).
A permutation that is not decomposable is called {\em indecomposable}.
These permutations are also known under the terms
{\em connected} and {\em irreducible}.
\begin{theorem}
Let $S$ be a subset of $\mathfrak{S}_{m+1}$ such that
each permutation in $S$ is
indecomposable
and the monotone permutation $(m+1) \cdots 2 1$
does not belong to $S$.
Then $\alpha_{n}(S)$ has the asymptotic expression
$\alpha_{n}(S)/n!
  =
      c \cdot \lambda^{n}  + O(r^{n})$,
where $c$, $\lambda$ and $r$
are positive constants such that $\lambda > r$.
\label{theorem_indecomposable_permutations}
\end{theorem}

More generally we can characterize the spectrum of $T$ in terms of
another graph associated to~$S$. To define it, let $\Delta_{\pi}$
for $\pi \in \mathfrak{S}_{m}$
denote the set of points
$x=\left( x_{1},\ldots,x_{m}\right) \in\left(0,1\right)^{m}$
with $x_{i}\neq x_{j}$ for $i\neq j$ and $\Pi(x)=\pi$. The graph $H_{S}$
has vertex set $\cup_{\pi\in\mathfrak{S}_{m}}\Delta_{\pi}$ and directed
edges from $\left( x_{1},\ldots,x_{m}\right)$
to $(x_{2},\ldots,x_{m+1})$
if $x_{1}\neq x_{m+1}$ and $\Pi(x_{1},\ldots,x_{m+1})\notin S$.
We show that if the graph $H_{S}$
is strongly connected
there is an upper bound on the length
of the directed path connecting
any two vertices.
We define
the {\em period} of a strongly connected graph $G$ as follows.
Fix a vertex~$v$ of~$G$ and,
for $k$ a non-negative integer, let $X_{k}$ be the set of all
vertices in $G$ that can be reached from $v$ in exactly $k$ steps.
The set $Q$ of
all $k$ with $v\in X_{k}$ is a semigroup and generates a subgroup
$d\mathbb{Z}$ of $\mathbb{Z}$.
The integer~$d$ is the {\em period} of the
graph $G$. Note that, if $G$ is strongly connected, then $G$ has
period $d$ for some positive integer~$d$.
Finally, a graph is called {\em ergodic}
if it is strongly connected and has period $1$.

\begin{theorem}
Suppose that $S$ is a set of forbidden
patterns and that the graph $H_{S}$ is strongly connected with
period $d$. Then the operator $T$ has positive spectral radius $r(T)$
and $T$ has a simple eigenvalue
$\lambda=r(T)$ with strictly positive eigenfunction. Moreover,
the spectrum of $T$ is invariant under multiplication by
$\exp(2 \cdot \pi \cdot i/d)$.
\label{theorem_period}
\end{theorem}

In case when the period is $1$, we have the conclusion:
\begin{theorem}
Suppose that $S$ is a set of forbidden patterns such that
the graph $H_{S}$ is ergodic. Then
the operator $T$ is positivity improving.
That is, the operator has a unique largest eigenvalue
which is simple, real and positive and
the associated eigenfunction is positive.
\label{theorem_graph_H_ergodic}
\end{theorem}

For certain explicit patterns, we can compute the spectrum and
eigenfunctions of $T$ and obtain sharp asymptotic formulas for
$\alpha_{n}(S)/n!$.

\begin{example}
{\rm
When $S$ is empty, directly we have $\alpha_{n}(S) = n!$
for all $n \geq 0$.
The associated operator only has one non-zero eigenvalue, namely $1$,
with eigenfunction and adjoint eigenfunction $\mathbf{1}$.
This is the only case we know where the number of eigenvalues
is finite.
}
\end{example}

\begin{example}
{\rm
If $S=\left\{ 123\right\}$ we show that the operator $T$ has
a trivial kernel and spectrum given by
$\left\{\lambda _{k}\right\}_{k \in \mathbb{Z}}$ where
$$
\lambda _{k}=\frac{\sqrt{3}}{2 \cdot \pi \cdot \left( k+\frac{1}{3}\right) }.
$$
Furthermore all the eigenvalues are simple.
We also compute the eigenfunctions of $T$
and the adjoint operator $T^{*}$ and
obtain
\begin{equation}
\frac{\alpha_{n}(123)}{n!}
   =
\exp\left( \frac{1}{2 \cdot \lambda_{0}}\right) \cdot \lambda_{0}^{n+1}
   +
O\left( |\lambda_{-1}|^{n}\right)
\label{equation_123}
\end{equation}
where $\lambda_{0}=r(T)$
and $\lambda_{-1}$ is the next largest eigenvalue
in modulus.
For more terms in the asymptotic expansion
see Theorem~\ref{theorem_123}.
}
\end{example}

\begin{example}
{\rm
If $S=\left\{ 213\right\}$, we show that
the nonzero eigenvalues of the operator $T$
are the roots of the equation
$$
\erf\left( \frac{1}{\sqrt{2} \cdot \lambda}\right) =\sqrt{\frac{2}{\pi }}
$$
which has the unique real root
$\lambda_{0}=0.7839769312\ldots$.
Moreover, $\lambda_{0}$ is the largest root in modulus
of the equation. We then have
\begin{equation}
\frac{\alpha_{n}(213)}{n!}
    =
\exp\left( \frac{1}{2 \cdot \lambda_{0}^{2}}\right) \cdot \lambda_{0}^{n+1} 
    +
O\left( |\lambda_{1}|^{n}\right)
\label{equation_213}
\end{equation}
where
$\lambda_{1,2} = 0.2141426360\ldots \pm 0.2085807022\ldots\cdot i$
are the next two largest roots of the eigenvalue equation.
See~Section~\ref{section_213} for the calculations.
}
\end{example}

\begin{example}
{\rm
If $S = \left\{123,321\right\}$, the numbers $\alpha_{n}(S)$ are
given by $\alpha_{n}(S)=2E_{n}$ for $n \geq 2$ where $E_{n}$ is the
$n$th Euler number. In this case, we can use spectral methods to
obtain the
classical convergent expansion
\begin{equation}
\frac{E_{n}}{n!} 
   =
 2\cdot
\sum_{\substack{ j\geq 1 \\ j \text{ \scriptsize odd}}}
\left( -1 \right)^{\frac{j-1}{2}(n+1)}
  \cdot
\left( \frac{\pi \cdot j}{2} \right)^{-n-1}.
\label{equation_Euler}
\end{equation}
This formula was derived by
Ehrenborg, Levin, and Readdy~\cite{Ehrenborg_Levin_Readdy}
(Corollary 4.2) by using Fourier series. In this case the spectrum of the
operator $T$ is real and invariant under the reflection
$\lambda \mapsto -\lambda$; in particular the two
largest eigenvalues are $\pm 2/\pi$.
Moreover, the matrix $U$
introduced in Proposition~\ref{proposition_T2},
is similar to the matrix
$\begin{pmatrix} 0 & 1 \\ 1 & 0 \end{pmatrix}$
for a cyclic permutation of order two.
}
\end{example}

\begin{example}
{\rm
Let $S = \left\{123,231,312\right\}$.
In this case the asymptotic expansion
converges and we can conclude that
the number of such permutations
in ${\mathfrak S}_{n}$ is given by $n \cdot E_{n-1}$,
a result due to Kitaev and Mansour~\cite{Kitaev_Mansour}.
See Section~\ref{section_123_231_312}.
}
\end{example}

It is easy to find examples of patterns $S$ for which $\rho(T)=0$.

\begin{example}
{\rm
Let $S = \left\{132,231\right\}$. An $S$-avoiding
permutation has no peaks (viewed as the graph of a function from
$\left\{1, \ldots, n\right\}$ to itself) and it is easy
to see that $\alpha_{n}(S)=2^{n-1}$.
However, observe that the proportion
$\alpha_{n}(S)/n!$ is subexponential.
It is straightforward to verify
that the operator $T$ has
no non-zero eigenvalues.
Also note that the graph $G_{S}$ is not strongly connected.
}
\label{example_132_231}
\end{example}

\begin{example}
{\rm
Let $S = \left\{123,213,231,321\right\}$. The directed
graph $G_{S}$ is strongly connected but the monotone permutations $123$ and
$321$ are excluded. In this case $\alpha_{n}(S)=2$ for all $n \geq 2$.
}
\label{example_123_213_21_321}
\end{example}

We close our introduction with a brief overview on the subject of
pattern avoidance in permutations (for more details we refer
to~\cite{Bona,kit}).  The ``classical'' definition of a pattern is
slightly different than the one provided above. We say that a
permutation~$\pi$ avoids a pattern $\sigma$ if $\pi$ does not contain
a {\em subsequence} which is order-isomorphic to $\sigma$. The study
of such patterns originated in theoretical computer science by
Knuth~\cite{Knuth}. However, the first systematic study was done
by Simon and Schmidt~\cite{Simion_Schmidt}, who completely classified the
avoidance of patterns of length three. Since then several hundred
papers related to the field have been published.

One of the most important results in the subject is the proof by
Marcus and Tardos~\cite{Marcus_Tardos}
of the so-called Stanley--Wilf conjecture related to
the asymptotic behavior of the number of permutations that avoid a given
pattern. It states that for any permutation $\sigma$ there exists a constant
$c$ (depending on $\sigma$) such that the number of the permutations of
length $n$ that avoid $\sigma$ is less than $c^{n}$.

In this paper we also study asymptotic behavior of permutations
avoiding patterns, but we consider {\em consecutive} patterns,
occurrences of which correspond to (contiguous) factors, rather than
subsequences, anywhere in permutations. Simultaneous avoidance of
consecutive patterns of length 3 is studied in~\cite{Kitaev}
using direct combinatorial arguments.
Other approaches to study consecutive patterns
were introduced recently which include considering increasing binary
trees~\cite{Elizalde_Noy} by Elizalde and Noy,
symmetric functions~\cite{Mendes_Remmel} by Mendes and Remmel,
homological algebra~\cite{Dotsenko_Knoroshkin}
by Dotsenko and Knoroshkin,
and graphs of pattern overlaps~\cite{Avgustinovich_Kitaev}
by Avgustinovich and Kitaev (see Chapter 5 of \cite{kit} entirely dedicated to known methods and results on consecutive patterns).

In~\cite{Elizalde_Noy} asymptotics for
the following consecutive patterns is given: 123, 132, 1342, 1234
and 1243. These results are obtained by representation of
permutations as increasing binary trees, then using symbolic methods
followed by solving certain linear differential equations with
polynomial coefficients to get corresponding exponential
generating functions, and, finally, using the following result
(see~\cite[Theorem IV.7]{Flajolet_Sedgewick} for a discussion).
\begin{theorem}
If $f(z)$ is analytic at $0$ and $R$ is the 
modulus of a singularity nearest to the origin
in the sense that
$$    R = \sup\{r \geq 0 : f \text{ is analytic in } |z|<r \}, $$
then the coefficient $f_{n} = [z^{n}] f(z)$ satisfies
$$ \limsup_{n \rightarrow \infty} |f_{n}|^{1/n} = 1/R . $$
\end{theorem}
This differs from our method, as we
obtain detailed asymptotic expansions that allows for explicit
computation of leading coefficients in many cases.
As special cases of our results,
we get more detailed asymptotics for some of the results of Elizalde and
Noy~\cite{Elizalde_Noy}. Also note
that the associated generating function is related to our
operator $T$ by the identity
$$
\sum_{n \geq 0} \alpha_{n}(S) \cdot \frac{z^{n}}{n!}
  =
1 + \cdots + z^{m-1}
 + z^{m} \cdot \pair{(I - z T)^{-1}(\mathbf{1})}{\mathbf{1}} .
$$
{}From this identity it is clear that the radius of convergence of the
generating function is determined by the spectrum of $T$.

The underlying principle that makes our method work
is that one can pick a permutation in $\mathfrak{S}_{n}$
at random with uniform distribution, by picking
a point $(x_{1}, \ldots, x_{n})$ in the unit cube $[0,1]^{n}$
and applying the function $\Pi$.
This method has been used in~\cite{Ehrenborg_Levin_Readdy}
to obtain quadratic inequalities
for the descent set statistics
and in~\cite{Ehrenborg_Farjoun} to enumerate
alternating $2$ by $n$ arrays.

\section{The operator $T$}
\label{section_T}

We now begin our study of the operator $T$.

\subsection{Connection with pattern avoidance}
\label{subsection_T_avoid}

Recall that $S$ is a collection of forbidden patterns
of length $m+1$, that is, $S$ is a subset of ${\mathfrak S}_{m+1}$.
The function $\chi$ is defined on the unit cube $[0,1]^{m+1}$ by
$$
\chi(x)
   =
   \left\{ \begin{array}{c l}
             1 & \Pi(x) \notin S, \\
             0 & \Pi(x) \in S,
           \end{array} \right.
$$
and the operator $T$ on $L^{2}([0,1]^{m})$ by
$$
(T f)(x_{1},\ldots,x_{m})
    =
\int_{0}^{1}
   \chi(t,x_{1},\ldots,x_{m})
      \cdot
   f(t,x_{1},\ldots,x_{m-1}) dt .
$$
For $n \geq m$, define $\chi_{n}$ on
the $n$-dimensional cube $[0,1]^{n}$ by
\begin{equation}
\chi_{n}(x_{1},\ldots,x_{n})
   =
\prod_{j=1}^{n-m} \chi(x_{j},\ldots,x_{m+j}) .
\label{equation_chi_n}
\end{equation}
This allows us to express powers of our operator $T$.
If $k<m$ then
\begin{equation}
\left(T^{k}f\right)(x)
      =
\int_{\left[0,1\right]^{k}}
    \chi_{m+k}(t_{1},\ldots,t_{k},x_{1},\ldots,x_{m})
      \cdot
    f(t_{1},\ldots,t_{k},x_{1},\ldots,x_{m-k})
    dt_{1}\cdots dt_{k}
\label{equation_T_k_1}
\end{equation}
while if $k\geq m$,
\begin{equation}
\left(T^{k}f\right)(x)
      =
\int_{\left[0,1\right]^{k}}
    \chi_{m+k}(t_{1},\ldots,t_{k},x_{1},\ldots,x_{m})
      \cdot
    f(t_{1},\ldots,t_{m})
    dt_{1}\cdots dt_{k}
\label{equation_T_k_2}
\end{equation}

\begin{proposition}
The number of $S$-avoiding permutations in ${\mathfrak S}_{n}$
for $n \geq m$ is given by
$\alpha_{n}(S) = n! \cdot \pair{T^{n-m}(\mathbf{1})}{\mathbf{1}}$.
\end{proposition}
\begin{proof}
Equation~(\ref{equation_chi_n}) allows us to conclude that
$$ \chi_{n}(x)
            =
            \left\{\begin{array}{c l}
                      1 & \text{ if } \Pi(x) \text{ avoids } S, \\
                      0 & \text{ if } \Pi(x) \text{ does not avoid } S.
                   \end{array} \right.
$$
Hence by integrating over the $n$-dimensional cube
$$
\int_{\left[0,1\right]^{n}}\chi_{n}(x) dx = \frac{\alpha_{n}(S)}{n!},
$$
since each simplex in the standard triangulation of $\left[0,1\right]^{n}$
corresponds to a permutation $\pi\in\mathfrak{S}_{n}$ and each such simplex
has volume $1/n!$. Using this observation and the
identities~(\ref{equation_T_k_1}) and~(\ref{equation_T_k_2}), we can rewrite
the integral as $\pair{T^{n-m}(\mathbf{1})}{\mathbf{1}}$.
\end{proof}

\begin{lemma}
The operator $T$ is a bounded operator
with norm at most $1$.
\end{lemma}
\begin{proof}
Apply the Cauchy--Schwarz inequality to equation~(\ref{equation_T})
in the variable $t$ to obtain
\begin{eqnarray*}
    & &
\left|
 \int_{0}^{1}
      \chi(t,x_{1},\ldots,x_{m})
    \cdot
      f(t,x_{1},\ldots,x_{m-1}) dt
\right|^{2} \\
    & \leq &
 \int_{0}^{1}
    |\chi(t, x_{1},\ldots,x_{m})|^{2} dt
  \cdot
 \int_{0}^{1}
    |f(t,x_{1},\ldots,x_{m-1})|^{2} dt \\
    & \leq &
 \int_{0}^{1}
    |f(t,x_{1},\ldots,x_{m-1})|^{2} dt .
\end{eqnarray*}
Now integrating over the variables $x_{1}, \ldots, x_{m}$,
we obtain that
$\|T(f)\|^{2} \leq \|f\|^{2}$ proving the bound.
\end{proof}

\begin{lemma}
The adjoint operator $T^{*}$ is given by
$$
\left( T^{*}f\right)(x)
   =
\int_{0}^{1} \chi(x_{1},\ldots,x_{m},u)f(x_{2},\ldots,x_{m},u) du .
$$
\label{lemma_adjoint}
\end{lemma}
\begin{proof}
Since $\chi$ is real-valued, we have that
\begin{eqnarray*}
\pair{T(f)}{g}
  & = &
\int_{[0,1]^{m}} 
\int_{0}^{1} \chi(t,x_{1},\ldots,x_{m}) \cdot f(t,x_{1},\ldots,x_{m-1}) dt
\cdot
\overline{g(x_{1},\ldots,x_{m})} dx_{1} \cdots dx_{m} \\
  & = &
\int_{[0,1]^{m}} 
f(t,x_{1},\ldots,x_{m-1}) dt
\cdot
\overline{\int_{0}^{1} \chi(t,x_{1},\ldots,x_{m}) \cdot 
           g(x_{1},\ldots,x_{m}) dx_{m}}
dt dx_{1} \cdots dx_{m-1} ,
\end{eqnarray*}
proving the lemma.
\end{proof}

The spectrum of an adjoint operator $A^{*}$ is
given by conjugate of the spectrum of $A$.
That is, if $\lambda$ is an eigenvalue of $A$,
then $\overline{\lambda}$ is an eigenvalue of $A^{*}$.
However, for our operator $T$, the complex eigenvalues
come in conjugate pairs.
This is proved using that $\chi$ is real-valued.
\begin{lemma}
We have that
$\overline{T(f)} = T(\overline{f})$
and
$\overline{T^{*}(f)} = T^{*}(\overline{f})$.
Hence
if $\lambda$ is an eigenvalue of $T$ with
eigenfunction $\varphi$,
then 
$\overline{\lambda}$ is also eigenvalue of $T$ with
eigenfunction $\overline{\varphi}$.
Similarly,
if $\lambda$ is an eigenvalue of $T^{*}$ with
eigenfunction $\psi$,
then 
$\overline{\lambda}$ is also eigenvalue of $T^{*}$ with
eigenfunction $\overline{\psi}$.
\end{lemma}

\subsection{Eigenvalues and asymptotic expansion}

Since the operator $T$ is bounded, we can use
spectral analysis in order
to explore the operator $T$.
We refer the reader to
Dunford and Schwarz~\cite[Chapter~VII]{Dunford_Schwarz}
for a more detailed exposition.

As we defined in the introduction,
the resolvent set $\varsigma(T)$
is the set of complex numbers $z$ such 
the operator $(z I - T)^{-1}$ exists as a bounded operator.
The {\em spectrum} $\sigma(T)$ of $T$ is the complement of the
resolvent set $\varsigma(T)$.

The {\em index} of a complex number $\lambda$
is the smallest non-negative integer $\nu$ such
that the equation $(\lambda I - T)^{\nu+1} f = 0$
implies $(\lambda I - T)^{\nu} f = 0$ for all functions $f$.
Informally speaking, for operators on a finite-dimensional
vector spaces, the index of an eigenvalue is the size
of the largest Jordan block associated with
that eigenvalue.
A point $\lambda$ in the spectrum is called a pole $T$
of {\em order} $\nu$ if the function $R(z;T)$ has a pole
at $\lambda$ of order $\nu$.
Theorem 18 in~\cite[Section~VII.3]{Dunford_Schwarz} states
that the order of a pole $\lambda$ is equal to its index.

Define the operator $E(\lambda)$ by the integral
$$ 
     E(\lambda)
   =
     \frac{1}{2 \pi i}
        \cdot
     \ointctrclockwise_{C}
        \frac{1}{z I - T} dz  ,
$$
where $C$ is a positive oriented closed curve in the complex plane
only containing the eigenvalue $\lambda$ from the spectrum~$\sigma(T)$.
It follows from~\cite[Section~VII.3]{Dunford_Schwarz} that
$E(\lambda)$ is a projection.

\begin{lemma}
The operator $T^{m}$ is compact.
\end{lemma}
\begin{proof}
The operator $T^{m}$ has the form
$$
T^{m}(f) 
    =
\int_{\left[0,1\right]^{m}}
     \chi_{2m}(t_{1},\ldots,t_{m},x_{1},\ldots,x_{m})
         \cdot
     f(t_{1},\ldots,t_{m})
     dt_{1}\cdots dt_{m} .
$$
Since $\chi_{2m}$ is a bounded function
we conclude that
$T^{m}$ is a Hilbert--Schmidt operator,
and hence a compact operator.
\end{proof}

Using Theorems~5 and~6 in~\cite[Section~VII.4]{Dunford_Schwarz}
we conclude:
\begin{theorem}
The spectrum of $T$ is at most denumerable
and has no point of accumulation in the complex plane
except possibly $0$. Every non-zero number $\lambda$
in $\sigma(T)$
is a pole of $T$ and has finite positive index.
For such a number $\lambda$
the projection $E(\lambda)$
has a non-zero finite dimensional range
and it is given by
$\{f \in L^{2}([0,1]^{m}) \: : \:
   (\lambda I - T)^{\nu} f = 0\}$,
where $\nu$ is the order of the pole $\lambda$.
\end{theorem}

\vanish{
The eigenvalues of the adjoint operator $T^{*}$
are described by
Lemma~7 in~\cite[Section~VII.3]{Dunford_Schwarz}:
\begin{lemma}
The spectrum of the adjoint operator $T^{*}$
is the complex conjugate of the spectrum of the operator $T$.
\end{lemma}
}

Recall that an eigenvalue $\lambda$ is simple
if the range of $E(\lambda)$ 
is one-dimensional.
That is, the eigenvalue equation
$\lambda \varphi = T \varphi$
has a unique solution up to a scalar multiple
and the generalized eigenvalue equation
$\lambda f = T f + \varphi$
has no solution.

\begin{lemma}
Let $\lambda$ be a simple eigenvalue of the operator $T$
with associated eigenfunction $\varphi$.
Let $\psi$ be the eigenfunction of the adjoint operator $T$
with eigenvalue $\lambda$.
Then the projection $E(\lambda)$ is given by
$$   E(\lambda)(f)
   =
     \frac{\pair{f}{\overline{\psi}}}{\pair{\varphi}{\overline{\psi}}}
     \cdot \varphi   .  $$
\end{lemma}
\begin{proof}
Since the eigenvalue $\lambda$ is simple,
the range of the projection $E(\lambda)$ 
is one-dimensional and spanned by
the eigenfunction $\varphi$.
Since the projection is continuous
we may assume that it has the form
$E(\lambda)(f) = \pair{f}{\alpha} \cdot \varphi$
for some function $\alpha$.
It is straightforward to observe that since $E(\lambda)$ is a projection,
that is, $E(\lambda)^{2} = E(\lambda)$,
we have that $\pair{\varphi}{\alpha} = 1$.

Now the adjoint operator $E(\lambda)^{*}$ is
given by
$E(\lambda)^{*}(g) = \pair{g}{\varphi} \cdot \alpha$
since
$$\pair{g}{E(\lambda)(f)}
    =
  \pair{g}{\pair{f}{\alpha} \cdot \varphi}
    =
  \pair{\alpha}{f} \cdot \pair{g}{\varphi}
    =
  \pair{\pair{g}{\varphi} \cdot \alpha}{f}  .  $$
Hence $\alpha$ belongs to the range of the
adjoint operator, that is, it is multiple of
the eigenfunction $\overline{\psi}$.
In fact, we observe that
$\alpha$ is given by
$\overline{\pair{\varphi}{\overline{\psi}}}^{-1}
\cdot \overline{\psi}$.
\end{proof}

\begin{proof}[Proof of Theorem~\ref{theorem_expansion}]
By analytic functional calculus
we can evaluate the operator $T^{n-m}$ by
integrating in the complex plane;
see Theorem~6(c) in~\cite[Section~VII.3]{Dunford_Schwarz}.
We have
$$
T^{n-m}
   =
     \frac{1}{2 \pi i}
         \cdot
     \ointctrclockwise_{\left| z\right| = R}
          \frac{z^{n-m}}{z I - T} dz  , 
$$
where $R$ is greater than the spectral radius of $T$
and we orient the circle in positive orientation.

Let $\sigma$ be the set $\{\lambda_{1}, \ldots, \lambda_{k}\}$
and let $E(\sigma)$ denotes the sum of the projections
$E(\lambda_{1}) + \cdots + E(\lambda_{k})$.
By Theorem~22 in~\cite[Section~VII.3]{Dunford_Schwarz} 
and that
the eigenvalues $\lambda_{1}, \ldots, \lambda_{k}$
are simple,
we have that
$$   T^{n-m} \cdot E(\sigma)
   =
     \sum_{i=1}^{k} E(\lambda_{i}) \cdot \lambda_{i}^{n-m} . $$

We can estimate the operator
$T^{n-m} \cdot (I - E(\sigma))$
by shrinking the path of integration to a circle of radius $r$
$$
T^{n-m} \cdot (I - E(\sigma))
   =
\frac{1}{2 \pi i}
         \cdot
     \ointctrclockwise_{\left| z\right| = r}
          \frac{z^{n-m}}{z I - T} dz  . 
$$
We bound this integral by
\begin{eqnarray*}
\left\| T^{n-m} \cdot (I - E(\sigma)) \right\|
  & = &
\left\|
\frac{1}{2 \pi i}
         \cdot
     \ointctrclockwise_{\left| z\right| = r}
          \frac{z^{n-m}}{z I - T} dz  
\right\| \\
  & \leq &
\frac{1}{2 \pi}
         \cdot
     \ointctrclockwise_{\left| z\right| = r}
\left\|
          \frac{1}{z I - T}
\right\| dz
\cdot r^{n-m}  \\
  & \leq &
\sup_{\left| z\right| =r}
    \left\| \left( z I - T \right)^{-1}\right\|
\cdot r^{n-m}  \\
  & = &
O\left(r^{n}\right)  ,
\end{eqnarray*}
where the last equality follows from that the supremum does not
depend on $n$.
Hence the inner product 
$\pair{T^{n-m} \cdot (I-E(\sigma)) \mathbf{1}}{\mathbf{1}}$
is also bounded by $O(r^{n})$.
Thus we conclude that
\begin{eqnarray*}
\pair{T^{n-m}\mathbf{1}}{\mathbf{1}}
  & = &
\pair{T^{n-m} E(\sigma) \mathbf{1}}{\mathbf{1}}
    +
\pair{T^{n-m} \cdot (I - E(\sigma)) \mathbf{1}}{\mathbf{1}} \\
  & = &
\sum_{i=1}^{k} 
                \pair{E(\lambda_{i})\mathbf{1}}{\mathbf{1}}
              \cdot
                \lambda_{i}^{n-m} 
    +
O(r^{n}) \\
  & = &
\sum_{i=1}^{k}
     \frac{\pair{\varphi_{i}}{\mathbf{1}} 
           \cdot \pair{\mathbf{1}}{\overline{\psi_{i}}}}
          {\pair{\varphi_{i}}{\overline{\psi_{i}}}} 
   \cdot
     \lambda_{i}^{n-m}
    +
O(r^{n}) .
\end{eqnarray*}
\end{proof}

If the eigenvalue $\lambda$ is not real
then $\overline{\lambda}$ is also an eigenvalue.
The two terms corresponding to these two eigenvalues
can be combined as follows.
Let $\lambda = r \cdot e^{i \cdot \theta}$
and 
$\pair{\varphi}{\mathbf{1}} 
    \cdot
 \pair{\mathbf{1}}{\overline{\psi}}
    /
 \pair{\varphi}{\overline{\psi}}
=
  s \cdot e^{i \cdot \beta}$.
Then we have 
\begin{eqnarray*}
     \frac{\pair{\varphi}{\mathbf{1}} 
           \cdot \pair{\mathbf{1}}{\overline{\psi}}}
          {\pair{\varphi}{\overline{\psi}}} 
   \cdot
     \lambda^{n-m}
  +
     \frac{\pair{\overline{\varphi}}{\mathbf{1}} 
           \cdot \pair{\mathbf{1}}{\psi}}
         {\pair{\overline{\varphi}}{\psi}} 
   \cdot
     \overline{\lambda}^{n-m}
  & = &
2 \cdot
\Re\left(
      s \cdot e^{i \cdot \beta}
        \cdot
      r^{n-m} \cdot e^{i \cdot (n-m) \cdot \theta}
\right) \\
  & = &
      2 \cdot
      s \cdot r^{n-m} 
        \cdot
      \cos( \beta + (n-m) \cdot \theta ) .
\end{eqnarray*}

\subsection{Bounds on the norm and spectral radius}

\begin{proposition}
\label{proposition_T1}
Let $T$ be an operator of the form~(\ref{equation_T}) and suppose
that $T$ has a positive spectral radius $r(T)$.
Then $T$ has a positive eigenvalue $\lambda=r(T)$ with
non-negative eigenfunction $\varphi$.
Furthermore, the eigenfunction $\varphi$ is almost everywhere positive.
Similarly, the adjoint operator $T^{*}$ also has
$\lambda=r(T)$ as an eigenvalue with
an almost everywhere positive eigenfunction $\psi$.
\end{proposition}
\begin{proof}
Let $K$ be the cone of non-negative functions in $L^{2}([0,1]^{m})$.
Since $T$ preserves this cone and
$T$ has nonzero spectral radius, it follows from Theorem~6.1 of
\cite{Krein_Rutman} and the fact that the cone of
positive functions is self-dual that
$T$ and $T^{*}$ both have $\lambda=r(T)$ as an eigenvalue with at least one
strictly positive eigenfunction.
\end{proof}

For an example of an eigenfunction $\varphi$ corresponding to the
largest eigenvalue, taking the value~$0$ on a set of measure~$0$,
see Proposition~\ref{proposition_123_231_312_eigenfunction}
where $123,231,312$-avoiding permutations are discussed.

\begin{lemma}
For $k \geq m$ the norm of the operator $T^{k}$ is bounded above by
$\sqrt{a_{m+k}(S)/(m+k)!}$.
\end{lemma}
\begin{proof}
Applying the Cauchy--Schwarz inequality to equation~(\ref{equation_T_k_2})
in the variables $t_{1}$ through $t_{k}$,
we obtain
\begin{eqnarray*}
    & &
\left|
 \int_{[0,1]^{k}}
      \chi_{m+k}(t_{1},\ldots,t_{k},x_{1},\ldots,x_{m})
    \cdot
      f(t_{1},\ldots,t_{m}) dt
\right|^{2} \\
    & \leq &
 \int_{[0,1]^{k}}
    |\chi_{m+k}(t_{1},\ldots,t_{k},x_{1},\ldots,x_{m})|^{2} dt
  \cdot
 \int_{[0,1]^{k}}
    |f(t_{1},\ldots,t_{m})|^{2} dt \\
    & = &
 \int_{[0,1]^{k}}
    \chi_{m+k}(t_{1},\ldots,t_{k},x_{1},\ldots,x_{m}) dt
  \cdot
 \int_{[0,1]^{m}}
    |f(t_{1},\ldots,t_{m})|^{2} dt .
\end{eqnarray*}
Now integrating the variables $x_{1}, \ldots, x_{m}$
over $[0,1]^{m}$ we
have $\|T^{k}(f)\|^{2} \leq 
{\alpha_{m+k}(S)}/{(m+k)!} \cdot \|f\|^{2}$,
proving that $\|T^{k}\| \leq \sqrt{\alpha_{m+k}(S)/(m+k)!}$.
\end{proof}

\begin{proposition}
\label{proposition_key}
For a nonempty set of forbidden patterns $S$ we have
that spectral radius of the operator $T$ is less
than $1$, that is, $r(T)<1$.
\end{proposition}
\begin{proof}
We have 
$\alpha_{2m}(S) < (2 \cdot m)!$
implying that
$\left\| T^{m} \right\| < 1$.
The inequality
$\left\| T^{m \cdot n}\right\|
  \leq
\left\| T^{m} \right\|^{n}$
implies that
$r(T^{m})
 =\lim_{n\longrightarrow\infty}\left\| T^{m \cdot n}\right\|^{1/n}<1$
strictly.
The result follows by taking the $m$th root.
\end{proof}

\begin{proposition}
\label{proposition_T2}
Let $T$ be an operator of the form~(\ref{equation_T}) and suppose
that $\rho = r(T)>0$.
Then the operator $T$ admits a decomposition of the form
$$   T  =  \rho \cdot U  +  W  $$
where $U W = W U = 0$,
$U$ has finite-dimensional range, $U$ maps the interior of
the cone of positive functions into itself, the operator $W$
has spectral radius less than $\rho$, and the
eigenvalues of $U$ are roots of unity including $1$.
\end{proposition}
\begin{proof}
The orthogonal decomposition follows from Theorem~8.1 of \cite{Krein_Rutman}
applied to the operator $A=\rho^{-1} \cdot T$.
\end{proof}

In particular, the leading behavior of powers $T^{n}$
is determined by the spectral radius $r(T)$
and the finite-rank operator~$U$.

\begin{theorem}
Let $S$ be a set of forbidden patterns.
Then spectral radius of the operator $T$ is given by
$$
r(T)
  =
\lim_{n \longrightarrow \infty}
         \left( \frac{\alpha_{n}(S)}{n!} \right)^{1/n}  .
$$
\label{theorem_easy}
\end{theorem}
\begin{proof}
Suppose first that $r(T)=0$. From the inequality
$\left| \pair{T^{n}\mathbf{1}}{\mathbf{1}} \right|
\leq
\left\| T^{n}\right\|$ we immediately conclude that
$\left(\alpha_{n}(S)/n!\right)^{1/n}$
tends to $0$ as $n$ goes to infinity.
If $r(T)>0$, then by Proposition~\ref{proposition_T2},
we have
$$
\frac{\alpha_{n}(S)}{n!}
  =
  r(T)^{n} \pair{U^{n}\mathbf{1}}{\mathbf{1}}
  +
  \pair{W^{n}\mathbf{1}}{\mathbf{1}}
$$
where the second term obeys the estimate
$$
\left| \pair{W^{n}\mathbf{1}}{\mathbf{1}} \right|
   \leq
\left(r(T)-\varepsilon\right)^{n}
$$
for some $\varepsilon>0$ and all sufficiently large $n$. Moreover,
$\pair{U^{n}\mathbf{1}}{\mathbf{1}}$
is periodic in $n$ and strictly
positive since $U$ maps the interior of the cone of positive functions into
itself. Thus
$\pair{U^{n}\mathbf{1}}{\mathbf{1}}$
is both bounded above and below
by strictly positive constants.
It follows that
$\lim_{n\longrightarrow\infty}
\pair{U^{n}\mathbf{1}}{\mathbf{1}}^{1/n}=1$
so that
$\lim_{n\longrightarrow\infty}\left( \alpha_{n}(S)/n!\right)^{1/n}
    = r(T)$
as claimed.
\end{proof}

This result extends Theorem~4.1 of~\cite{Elizalde}, where consecutive
patterns consisting of a single permutation were considered. Moreover,
Theorem~\ref{theorem_easy} characterizes
the limit in terms of a spectral quantity
which can be computed in many cases of interest by solving the eigenvalue
problem for the integral operator $T$.

\section{Associated graphs}

\subsection{The directed graph $H_{S}$}

In this section we study the spectrum of $T$ using the infinite graph
$H_{S}$ described in the introduction.
Recall that $\Delta_{\pi}$ denotes the
open subset of $(0,1)^{m}$ with $x_{i}\neq x_{j}$ for $i\neq j$ and
$x_{i} < x_{j}$ if and only if $\pi(i)<\pi(j)$, and let
$$
X = \bigcup_{\pi\in\mathfrak{S}_{m}} \Delta_{\pi} .
$$
Thus the complement of $X$ consists of those points $x$ with
$x_{i} = x_{j}$ for at least one pair of distinct indices $i$ and $j$
and hence is a set of measure zero.
The graph $H_{S}$ has vertex set~$X$.
Recall that the directed edges of $H_{S}$
connected points $x$ and $y$ in $X$ with $x_{j+1} = y_{j}$
for $1 \leq j\leq m-1$,
$x_{1} \neq y_{m}$, and $\Pi(x_{1},\ldots,x_{m},y_{m}) \notin S$. It
follows from the definition that $\chi(t,x_{1},\ldots,x_{m}) = 1$
if and only
if there is a directed edge from $(t,x_{1},\ldots,x_{m-1})$ to
$(x_{1},\ldots,x_{m})$. That is, the function $\chi$ encodes the edge
information of the graph~$H_{S}$.

The next lemma connects the graph $H_{S}$ to mapping properties of the
operator~$T$. It will be used to show that the operator
is positivity improving.

\begin{lemma}
Suppose that $x,y \in X$ and that there is a directed
path from $x$ to $y$ of length $k \geq m$. Suppose further that $f$ is a
non-negative continuous function such that $f$ is non-zero in a neighborhood
of $x$. Then $(T^{k}f)(y) > 0$.
\label{lemma_T_graph}
\end{lemma}
\begin{proof}
Assume that the directed path is
$$ x = (x_{1},\ldots,x_{m}) 
           \longrightarrow
       (x_{2},\ldots,x_{m+1})
           \longrightarrow
       \cdots
           \longrightarrow
       (x_{k+1},\ldots,x_{k+m}) = y . $$
Let $\varepsilon$
be the minimum of the following finite set
$$
\left\{ \left| x_{i}-x_{j}\right| :
          1 \leq i < j \leq k+m,~j-i \leq m \right\}
\cup
\left\{x_{i}, 1-x_{i} :  1 \leq i \leq k+m \right\}  .
$$
Observe that $\varepsilon>0$ by the definition of $X$. 
Let $\delta = \varepsilon/3$.
For $s_{i}\in\left[
x_{i}-\delta,x_{i}+\delta\right] $, $1\leq i\leq k+m$, we have
that
$$ (s_{1},\ldots,s_{m})
       \longrightarrow
   (s_{2},\ldots,s_{m+1})
       \longrightarrow
   \cdots
       \longrightarrow
   (s_{k+1},\ldots,s_{k+m}) $$
is also a directed path in $H_{S}$. It follows
that $\chi_{k+m}(s_{1},\ldots,s_{k+m})=1$ for all such $s$.
Using~(\ref{equation_T_k_2}) we may estimate
\begin{eqnarray*}
\left( T^{k}f\right) (y)
 & \geq &
\int_{x_{1}-\delta}^{x_{1}+\delta} \cdots
\int_{x_{k}-\delta}^{x_{k}+\delta}
f(t_{1},\ldots,t_{m}) dt_{1} \cdots dt_{k} \\
& = &
\left( 2\delta\right)^{k-m}
  \int_{x_{1}-\delta}^{x_{1}+\delta} \cdots
  \int_{x_{m}-\delta}^{x_{m}+\delta}
f(t_{1},\ldots,t_{m}) dt_{1}\cdots dt_{m} \\
& > & 0,
\end{eqnarray*}
where in the last step we have used the positivity of $f$ in a neighborhood
of $(x_{1},\ldots,x_{m})$.
\end{proof}

\begin{proposition}
\label{proposition_period}
Suppose that $H_{S}$ is strongly connected with period $d$.
Then there is a decomposition
$$
X = {\displaystyle\bigcup_{i=0}^{d-1}} Y_{i}
$$
of $X$ into disjoint sets $Y_{i}$ with the property that
$T : L^{2}(Y_{i}) \longrightarrow L^{2}(Y_{i+1})$, where $Y_{d}=Y_{0}$.
\end{proposition}

\begin{proof}
Pick a base vertex $v$ of $H_{S}$. Let $X_{k}$ be the set of all vertices in
$H_{S}$ that can be reached from $v$ in $k$ steps, and let $Q$ be the subset
of the non-negative integers defined by
$$ Q = \left\{ k \: : \: v \in X_{k} \right\} . $$
Then $Q$ is a semigroup under addition and generates a subgroup of the
integers $\mathbb{Z}$. A subgroup of $\mathbb{Z}$ has the form $d\mathbb{Z}$
for some positive integer $d$; in this case, $d$ is the period of the graph
$H_{S}$.

Now define
$$ Y_{i} = \bigcup_{j:~j\equiv i \bmod d} X_{j} . $$
Observe that every directed edge in the graph $H_{S}$ goes from some $Y_{i}$
to the next~$Y_{i+1}$ (with addition modulo $d$). Also, observe that the
sets $Y_{i}$ are pairwise disjoint.

We claim that each $Y_{i}$ is open. To see this, suppose that $y\in Y_{i}$.
Pick a path from some vertex $x$ to the vertex $y$ having length greater
than $2m$. We can perturb this path in a small neighborhood of~l
$y$ using a
variant of the argument used at the beginning of Lemma~\ref{lemma_T_graph}
and conclude that $Y_{i}$ is open.

Now pick a permutation $\pi\in\mathfrak{S}_{m}$.
Since the sets $Y_{0}, \ldots, Y_{d-1}$ are pairwise disjoint
we have that
$\Delta_{\pi} = \Delta_{\pi} \cap X
              = \cup_{i=0}^{d-1} (\Delta_{\pi} \cap Y_{i})$.
Note $\Delta_{\pi}$ is a connected set
and $\Delta_{\pi} \cap Y_{i}$ are all open.
A connected set can only be the disjoint union
of one open set and hence there exists a unique index $i$
such that $\Delta_{\pi} \subseteq Y_{i}$.
Hence each set $Y_{i}$ is the disjoint
union of the sets $\Delta_{\pi}$.

Finally, suppose that $f$ is a continuous function on
$\left[0,1\right]^{m}$ with support in the set $Y_{i}$. We claim
that $T f$ is supported in the next set $Y_{i+1}$.
To see this, note that
$\chi(t,x_{1},\ldots,x_{m}) = 1$ if and only if there is a directed
edge from $(t,x_{1},\ldots,x_{m-1})$ to $(x_{1},\ldots,x_{m})$.
Assuming that $f(t,x_{1},\ldots,x_{m-1})$ is non-zero
where
$(t,x_{1},\ldots,x_{m-1})$ belongs to the set $Y_{i}$.
Since $f$ is continuous function,
$f$ is supported in a neighborhood
of the point $(t,x_{1},\ldots,x_{m-1})$.
By applying the definition of the operator $T$,
we have the function $T f$ supported
in a neighborhood of 
the point $(x_{1},\ldots,x_{m})$
in the next set $Y_{i+1}$.
That is, $T f$ is supported in $Y_{i+1}$.
\end{proof}

The next lemma is a straightforward consequence of
the definition of the graph $H_{S}$
and hence the proof is omitted.

\begin{lemma}
Suppose that $\alpha:(0,1) \longrightarrow (0,1)$ is a strictly increasing
function. If 
$$ (x_{1},\ldots,x_{m})
         \longrightarrow
   (x_{2},\ldots,x_{m+1})
         \longrightarrow
       \cdots
         \longrightarrow
   (x_{k+1},\ldots,x_{k+m}) $$
is a directed path in $H_{S}$,
then so is
$$ \left(\alpha(x_{1}),\ldots,\alpha(x_{m})\right)
         \longrightarrow
   \left(\alpha(x_{2}),\ldots,\alpha(x_{m+1})\right)
         \longrightarrow
       \cdots
         \longrightarrow
   \left(\alpha(x_{k+1}),\ldots,\alpha(x_{k+m})\right) . $$
\end{lemma}

We next show that there is an upper bound on
the length of the directed path between any two
vertices in the case when the graph $H_{S}$
is strongly connected.
\begin{proposition}
Suppose that $H_{S}$ is strongly connected with period $d$.
Then there is a positive integer $N$, a multiple of $d$, such that
for any two points $x$ and $z$ in the same component $Y_{i}$
there is a path from $x$ to $z$ in the graph $H_{S}$
of length $N$.
Especially, between any two vertices in the graph $H_{S}$
there is a directed path of length at most $N+d-1$.
\label{proposition_N}
\end{proposition}
\begin{proof}
For two vertices $x$ and $y$ of the graph $H$,
let $D(x,y)$ denote the length of the shortest path
from $x$ to $y$. Observe that this is not a distance
since it is not symmetrical in general.

For each permutation $\pi \in {\mathfrak S}_{m}$
pick a point $x_{\pi}$ in
$\Delta_{\pi}$,
such that all its coordinates are greater than~$1/2$.
Similarly, for all $0 \leq i \leq d-1$ pick
a point $y_{i}$ in $Y_{i}$, 
such that all its coordinates are less than~$1/2$.
Let $K$ denote the maximum of the finite set
$$   \{  D(x_{\pi}, y_{i})
           \: : \: 
         \Delta_{\pi} \subseteq Y_{i} \}  
  \:\: \cup \:\:
     \{  D(y_{i}, x_{\tau})
           \: : \: 
         \Delta_{\tau} \subseteq Y_{i} \}  .  $$
Note that $K$ is a multiple of the period $d$.
Let $Q_{i}$ denote the semigroup
$$ Q_{i} = \left\{ k \:\: : \:\:
   \text{ there is a path from $y_{i}$ to $y_{i}$ of exactly length $k$}
                  \right\} . $$
Note that $Q_{i} \subseteq d \cdot \mathbb{N}$.
Furthermore, there is no multiple $e$ of the period $d$
such that $Q_{i} \subseteq e \cdot \mathbb{N}$.
Hence there exists a positive integer $m_{i}$ such that
$d \cdot \mathbb{N} + m_{i} \subseteq Q_{i}$.
That is, there is path from $y_{i}$ to $y_{i}$ of any length
which is a multiple of $d$ and greater than or equal to $m_{i}$.
Let $M$ be the maximum of $m_{0}$ through $m_{d-1}$.
Note again that $M$ is a multiple of $d$.

Let $N = 2 \cdot K + M$. Pick two permutations
$\pi$ and $\tau$ such that
$\Delta_{\pi}, \Delta_{\tau} \subseteq Y_{i}$.
We claim that there is a path from
the point $x_{\pi}$ to the point $x_{\tau}$ of length $N$.
We find this path by picking three paths.
First, choose a path $p_{1}$ from $x_{\pi}$ to $y_{i}$ of length at most $K$.
Second, choose a path $p_{3}$ from $y_{i}$ to $x_{\tau}$ of length at most $K$.
Finally, pick a path $p_{2}$ from $y_{i}$ to $y_{i}$ of length
$M + 2 \cdot K - D(x_{\pi}, y_{i}) - D(y_{i}, x_{\tau})$
which is at least $M$. By concatenating the three paths
$p_{1}$, $p_{2}$ and $p_{3}$, the result follows.

Now let show that there is a path from any point
$x$ in $Y_{i}$ to any other point $z$ in $Y_{i}$
of length~$N$. 
Let $\pi$ and $\tau$ be the permutations
$\pi = \Pi(x)$ and $\tau = \Pi(z)$.
We do so using that we did choose
the coordinates of $x_{\pi}$ and $x_{\tau}$ greater
than the coordinates of $y_{i}$.
Namely, let $\varepsilon > 0$
be the smallest coordinate of the two points $x$ and $z$.
Let $y$ be the point $\varepsilon \cdot y_{i}$.
Now there is a monotone function $\alpha$ such that
$\alpha(x_{\pi}) = x$ and $\alpha(y_{i}) = y$.
Similarly, there is a monotone function $\beta$ such that
$\beta(y_{i}) = y$ and $\beta(x_{\tau}) = z$.
Concatenating the three paths
$\alpha(p_{1})$, $\varepsilon \cdot p_{2}$ and $\beta(p_{3})$
we obtain a path from $x$ to $z$ of length $N$.
\end{proof}

\begin{proof}[Proof of Theorem~\ref{theorem_period}]
It follows from Proposition~\ref{proposition_period}
that $T^{d}:L^{2}(Y_{i}) \longrightarrow L^{2}(Y_{i})$. We
will denote by $A$ the restriction of $T^{d}$ to $L^{2}(Y_{0})$. Choosing a
positive integer~$p$ with $p \cdot d \geq m$
we see that $A^{p}$ is a compact
operator from $L^{2}(Y_{0})$ to itself. The operator $A^{p}$ has discrete
spectrum which may accumulate only at $0$ so the same is true for $A$.
We claim that, if $\varphi$ is an eigenfunction of $A$ with nonzero
eigenvalue~$\lambda^{d}$,
then $T^{i}\varphi$ is a nonzero eigenfunction of $T^{d}$ with
eigenvalue $\lambda^{d}$ and support in~$Y_{i}$. The only nontrivial part
of this claim is that $T^{i}\varphi\neq 0$. To see this, note that
$T^{d-i}(T^{i}\varphi)=\lambda^{d} \cdot \varphi$
so $T^{i} \varphi$ cannot be zero since
$\lambda \neq 0$ and $\varphi$ is a nontrivial eigenfunction.

Suppose now that $\lambda^{d}\in\sigma(A)$,
let $\varphi$ be an eigenfunction of
$A$ corresponding to the eigenvalue $\lambda^{d}$, let $\omega$ be a $d$th
root of unity, and let
$$
\psi
    =
\varphi
    +
\frac{1}{\lambda \cdot \omega} \cdot T\varphi
    +
\cdots
    + \left( \frac{1}{\lambda \cdot \omega} \right)^{d-1}
      \cdot T^{d-1}\varphi .
$$
The function $\psi$ is nonzero because the right-hand terms are nonzero and
have disjoint supports. A direct computation shows that
$T\psi=\omega \cdot \lambda \cdot \psi$
so the spectrum of $T$ contains all of the numbers $\lambda \cdot \omega$
where $\lambda^{d}\in\sigma(A)$. On the other hand, any eigenvalue~$\mu$ of
$T$ gives rise to an eigenvalue $\mu^{d}$ of $T^{d}$, so the nonzero spectrum
of $T$ consists exactly of the numbers $\omega \cdot \lambda$
where $\omega$ is a
$d$th root of unity and $\lambda^{d}$ is an eigenvalue of $T^{d}$.

By combining Lemma~\ref{lemma_T_graph}
and Proposition~\ref{proposition_N}
we have that for a non-negative, but non-zero, function $f$
in $L^{2}(Y_{0})$ that
$T^{N}(f) = A^{N/d}(f)$ is a positive function.
Hence the operator $A$ is positivity improving
and by Kre\u{\i}n and Rutman~\cite[Theorem~6.3]{Krein_Rutman}
the operator $A$ has a positive spectral radius $r(A)$.
Furthermore, the spectral radius $r(A)$ is
a simple eigenvalue, all other eigenvalues are smaller
in modulus and the eigenfunction $\varphi$ corresponding to $r(A)$
is positive.
Hence the spectral radius of~$T$, $r(T) = \sqrt[n]{r(A)}$,
is positive and a simple eigenvalue of $T$. Finally,
the positive function $\varphi$ is also an eigenfunction of $T$
corresponding to the eigenvalue $r(T)$.
\end{proof}

\subsection{The directed graph $G_{S}$}

We examine relations between the infinite graph $H_{S}$
and the finite graph $G_{S}$. We need the following lemmas,

\begin{lemma}
Let $x = (x_{1}, \ldots, x_{m})$ be a vertex in $H_{S}$
such that $\Pi(x_{1}, \ldots, x_{m}) = \pi$.
Furthermore assume
there is an edge in $G_{S}$ labeled
$\tau$ leaving the vertex $\pi$.
Then there exists $x_{m+1}$ in $(0,1)$ such that
$\Pi(x_{1}, \ldots, x_{m+1}) = \tau$.
That is, there is an edge in $H_{S}$
leaving the vertex $x$.
\end{lemma}
\begin{proof}
Observe that $\tau(m+1)$ is bigger than exactly $\tau(m+1)-1$ of the
numbers $\tau(1), \ldots, \tau(m)$.
Hence pick $x_{m+1}$ such that it is bigger than exactly
$\tau(m+1)-1$ of the numbers $x_{1}, \ldots, x_{m}$.
\end{proof}

Iterating this lemma we obtain:
\begin{lemma}
Given a directed path from $\pi$
to $\sigma$ in the graph $G_{S}$.
Let $x$ be a vertex in $H_{S}$
such that $\Pi(x) = \pi$.
Then this directed path can be lifted
to a directed path in $H_{S}$
that ends with a vertex $y$ such that
$\Pi(y) = \sigma$.
\label{lemma_lifting}
\end{lemma}

Now we give a sufficient condition for $H_{S}$ to be strongly connected
in terms of the graph $G_{S}$.

\begin{proposition}
Let $S\subseteq\mathfrak{S}_{m+1}$, suppose that $G_{S}$ is strongly
connected, and suppose that the two monotone permutations $12\cdots (m+1)$
and $(m+1) \cdots 2 1$ do not belong to the set $S$.
Then the graph~$H_{S}$ is ergodic.
\label{proposition_G_to_H}
\end{proposition}
\begin{proof}
We first prove that $H_{S}$ is strongly connected.
Let $x=(x_{1},\ldots,x_{m})$ and $y=(y_{1},\ldots,y_{m})$ be two vertices of
$H_{S}$, and let $\pi=\Pi(x)$ and $\sigma=\Pi(y)$. Since $G_{S}$ is strongly
connected we can find a directed path from $\pi$ to $m\cdots 21$ in $G_{S}$.
This directed path lifts to a directed path from $x$ to
$z=(z_{1},\ldots,z_{m})$ in $H_{S}$, where $z_{1}>z_{2}>\cdots>z_{m}$.

For any directed graph $G$ define the reverse graph $G^{*}$ to the
graph $G$ where we reverse the direction of each edge. In the reverse graph
$G_{S}^{*}$ we have a directed path from $\sigma$ to $12\cdots m$. Lifting
this directed path to a directed path in $H_{S}^{*}$ and then reversing
the path, we obtain a directed path from $w=(w_{1},\ldots,w_{m})$ to $y$
in~$H_{S}$, where $w_{1}<\cdots<w_{m}$.

Finally, there is a directed path from $m\cdots 21$ to $12\cdots m$
in $G_{S}$.
Hence we know that there is a directed path from
$u=\left( u_{1},\ldots,u_{m}\right)$
to $v=(v_{1},\ldots,v_{m})$ in $H_{S}$, where
$u_{1}>\cdots>u_{m}$ and $v_{1}<\cdots<v_{m}$.

Choose $\alpha$ so that $0<\alpha\,<\min(z_{m},w_{1})$. We then have a
directed path from $\alpha\cdot u$ to $\alpha\cdot v$. Now, there is a
directed path from $z$ to $\alpha\cdot u$ of length $m$, namely
$$  z = (z_{1},\ldots,z_{m})
              \longrightarrow
        (z_{2},\ldots,z_{m},\alpha u_{1})
              \longrightarrow
           \cdots
              \longrightarrow
        (z_{m},\alpha u_{1},\ldots,\alpha u_{m-1})
              \longrightarrow
        (\alpha u_{1},\ldots,\alpha u_{m}) = \alpha \cdot u   $$
using the fact that $(m+1)\cdots 21$ is not forbidden. We can now
concatenate these five directed paths to obtain a path from $x$ via $z$, via
$\alpha\cdot u$, via $\alpha\cdot v$, via $w$, to $y$.

Since $G_{S}$ is strongly connected and has $m!$ vertices, an upper bound on
the length of the directed paths in $G_{S}$ chosen above is $m!-1$. Hence,
the path that we have constructed has length at most $3(m!-1)+2m$.

To observe that $H_{S}$ has period $1$
note that we can construct
a directed path from the vertex $x$ to the vertex
$y$ that has length one more than the above construction.
Namely, the path from $\pi$ to
$m\cdots 21$ can be extended by adding the loop
$(m+1) \cdots 2 1$ at the end.
Now by concatenating these two paths with a path from $y$ to $x$
we obtain two cycles whose lengths differ by one.
Since the greatest common divisor of two consecutive
integers is one, the graph $H_{S}$ is ergodic.
\end{proof}

Recall that a permutation $\pi$ in $\mathfrak{S}_{n}$
is {\em indecomposable}
if there is no index $i$ such that  $1 \leq i \leq n-1$
and $\pi(1), \ldots, \pi(i) \leq i$.
Otherwise the permutation is decomposable.
\begin{proposition}
Let $S \subseteq \mathfrak{S}_{m+1}$ such that
each permutation $\tau \in S$ is indecomposable.
Then the graph $G_{S}$ is strongly connected.
\label{proposition_decomposable}
\end{proposition}
\begin{proof}
Given two vertices $\pi = (\pi_{1}, \ldots, \pi_{m})$
and $\sigma = (\sigma_{1}, \ldots, \sigma_{m})$ of $G_{S}$.
Since the integers
$\pi_{1}, \ldots, \pi_{m}$, $\sigma_{1}+m, \ldots, \sigma_{m}+m$
are distinct,
the following path (described by its edges) is well defined:
$$ \Pi( \pi_{1}, \ldots, \pi_{m}, \sigma_{1}+m ),
   \Pi( \pi_{2}, \ldots, \pi_{m}, \sigma_{1}+m, \sigma_{2}+m ),
                   \ldots,
   \Pi( \pi_{m}, \sigma_{1}+m, \ldots, \sigma_{m}+m ) . $$
This path goes from the vertex $\pi$ to vertex $\sigma$.
Note that every edge $\tau^{\prime}$ on the path
is decomposable. Hence this path avoids
the forbidden indecomposable edges of $S$.
\end{proof}

\begin{proof}
[Proof of Theorem~\protect\ref{theorem_indecomposable_permutations}]
This follows directly from
Theorems~\ref{theorem_positivity_improving}
and~\ref{theorem_graph_H_ergodic}
and
Propositions~\ref{proposition_G_to_H}
and~\ref{proposition_decomposable}.
\end{proof}

\begin{proof}
[Proof of Corollary~\protect\ref{corollary_single_permutation}]
Let $\tau$ be the single permutation in the set $S$.
If $\tau$ is one of the two monotone permutations the result
will follow from descent pattern avoidance,
see~Theorem~\ref{theorem_D_H}.
If $\tau$ is indecomposable
the result follows from
Theorem~\ref{theorem_indecomposable_permutations}.
Finally if $\tau$ is decomposable
apply
Theorem~\ref{theorem_indecomposable_permutations}
to the upside down permutation
$(m+2-\tau(1), \ldots, m+2-\tau(m+1))$
which is not decomposable.
\end{proof}

Note that there are examples of patterns $S$ so that $H_{S}$ does not have
cycles even though the graph~$G_{S}$ has cycles.

\begin{example}
{\rm
$S = \left\{312,321\right\}$. In this case the graph $G_{S}$ has a
cycle. However, the graph $H_{S}$ does not have a cycle
and hence is not strongly connected. We observe this by
noting that for a directed edge $(x,y) \longrightarrow (y,z)$ in $H_{S}$ we
have that $x < \max(y,z)$. Hence none of $x_{i}$'s in a $k$-cycle
$(x_{1},x_{2}) \longrightarrow (x_{2},x_{3})  \longrightarrow \cdots
\longrightarrow (x_{k},x_{1})  \longrightarrow (x_{1},x_{2})$ can be the
largest.

Via the classical bijection $\pi \longmapsto \hat{\pi}$
(see~\cite[Section~1.3]{Stanley})
one obtains that the number $\{312,321\}$-avoiding
permutations are in bijection with involutions,
that is, permutations
$\pi$ such that $\pi^{2} = \operatorname{id}$. This was first observed by
Claesson~\cite{Claesson}. It follows that the generating function is
$\exp(z + z^{2}/2)$ and the asymptotic is $1/\sqrt{2} \cdot \exp(-1/4) \cdot
(n/e)^{n/2} \cdot \exp(\sqrt{n})$.
}
\label{example_312_321}
\end{example}

On the other hand, if $G_{S}$ does not have a cycle,
then neither does $H_{S}$.

\begin{lemma}
Let $S\subseteq\mathfrak{S}_{m+1}$ and suppose that $G_{S}$ has a directed
cycle that contains the two vertices $12\cdots m$ and $m \cdots 21$.
Moreover, assume that the two monotone permutations $12\cdots (m+1)$
and $(m+1) \cdots 21$ do not belong to the set $S$.
Then the graph $H_{S}$ contains a directed cycle.
\end{lemma}
\begin{proof}
Pick a vertex $x = (x_{1}, \ldots, x_{m})$ such that $\Pi(x) = 12 \cdots m$.
The directed path from the vertex $12 \cdots m$ to the vertex $m \cdots 21$
can be lifted to a path in $H_{S}$ from the vertex $x$ to a vertex $y =
(y_{1}, \ldots, y_{m})$ where $\Pi(y) = m \cdots 21$. Similarly, pick a
vertex $z = (z_{1}, \ldots, z_{m})$ such that $\Pi(z) = m \cdots 21$. The
directed path from the vertex $m \cdots 21$ to the vertex $12 \cdots m$ can
be lifted to a path in $H_{S}$ from the vertex $z$ to a vertex $w = (w_{1},
\ldots, w_{m})$ where $\Pi(w) = 12 \cdots m$.

Using the two monotone functions $\alpha, \beta: (0,1) \longrightarrow (0,1)$
defined by $\alpha(x) = (x+1)/2$ and $\beta(x) = x/2$, we have two directed
paths: one from $\alpha(x)$ to $\alpha(y)$
and one from $\beta(z)$ to $\beta(w)$.

Using that $(m+1) \cdots 21$ is an edge in $G_{S}$ we have the following
path from $\alpha(y)$ to $\beta(z)$, namely
$$ \alpha(y) = (\alpha(y_{1}), \ldots, \alpha(y_{m}))
                    \longrightarrow
               (\alpha(y_{2}), \ldots, \alpha(y_{m}), \beta(z_{1}))
                    \longrightarrow
                 \cdots
                    \longrightarrow
               (\beta(z_{1}), \ldots, \beta(z_{m})) = \beta(z)  . $$
Similarly, we have the directed path
$$ \beta(w) =  (\beta(w_{1}), \ldots, \beta(w_{m}))
                    \longrightarrow
               (\beta(w_{2}), \ldots, \beta(w_{m}), \alpha(x_{1}))
                    \longrightarrow
                 \cdots
                    \longrightarrow
               (\alpha(x_{1}), \ldots, \alpha(x_{m})) = \alpha(x) . $$
Concatenate these four directed paths to obtain a directed cycle
in $H_{S}$.
\end{proof}

\section{Computational techniques}
\label{section_techniques}

In this section we discuss descent pattern avoidance
which is a special case of pattern avoidance.
First we introduce an analogue of the de Bruijn graph
$D_{U}$, which has the advantage that it is smaller
than the graph $G_{S}$. Moreover, if the graph~$D_{U}$
is ergodic so is the graph $H_{S(U)}$
and we obtain that the associated operator is
positivity improving.
Second, for descent pattern avoidance
we obtain that the eigenfunctions has a simplified form.
Finally, in the last subsection
we consider pattern avoidance 
that has symmetry. In these cases we show that we can
obtain the adjoint eigenfunctions from the eigenfunctions.

\subsection{Descent pattern avoidance}
\label{subsection.descent}

The descent set of a permutation $\pi$ in the symmetric group on $n$
elements is the subset of $\{1,\ldots,n-1\}$, given by
$\{i:\pi_{i}>\pi_{i+1}\}$.
An equivalent notion is the descent word, defined as follows.
The descent word of the permutation $\pi$ is the word
$u(\pi) = u_{1}\cdots u_{n-1}$
where $u_{i}=a$ if $\pi_{i}<\pi_{i+1}$ and $u_{i}=b$ otherwise.

Let $U$ be a collection of $ab$-words of length $m$. The permutation $\pi$
avoids the set $U$ if there is no consecutive subword of the descent word of
$\pi$ contained in the collection $U$.

Descent pattern avoidance is a special case of consecutive pattern
avoidance. For instance, permutations avoiding the word $a a b$ is the
permutations avoiding the set $S = \{1243, 1342, 2341\}$, since these three
permutations are the permutations with descent word $a a b$.
More formally, for $U$ a subset of $\{a,b\}^{m}$
define $S(U) \subseteq \mathfrak{S}_{m+1}$
by
$$   S(U)
   =
     \{\pi \in \mathfrak{S}_{m+1} \:\: : \:\:
                  u(\pi) \in U                \}   .  $$
Then the set of permutations avoiding the descent words in $U$
is the set of permutations avoiding $S(U)$.

For $U$ a subset of $\{a,b\}^{m}$
define the associated {\em de Bruijn graph} $D_{U}$
by letting the vertex set be $\{a,b\}^{m-1}$. For $x,y \in \{a,b\}$
and $u \in \{a,b\}^{m-2}$ such that $x u y \not\in U$ let there be a
directed edge from $x u$ to $u y$. When the set $U$ is empty, the graph
$D_{U}$ is the classical de Bruijn graph $D_{m-1}$.

\begin{lemma}
Let $U$ be a subset of $\{a,b\}^{m}$.
If there is a cycle $c$ of length $N \geq n+1$
that do not consists only of the loop $a^{m}$
or not only of the loop $b^{m}$, then
the cycle~$c$ can be lifted to a cycle of length $N$
in the graph $H_{S(U)}$.
\label{lemma_cycles_in_D}
\end{lemma}
\begin{proof}
Let $c$ be the cycle
$$  v_{1} v_{2} \cdots v_{m-1} \longrightarrow
    v_{2} v_{3} \cdots v_{m}   \longrightarrow
                        \cdots \longrightarrow       
    v_{N} v_{1} \cdots v_{m-2} \longrightarrow       
    v_{1} v_{2} \cdots v_{m-1} ,  $$
where each $v_{i}$ is either $a$ or $b$.
We would like to pick $N$ real numbers
$x_{1}, \ldots, x_{N}$ in the 
interval $(0,1)$ such that
\begin{equation}
     \begin{array}{c c c}
 x_{i} < x_{i+1} & \text{ if } & v_{i} = a, \\
 x_{i} > x_{i+1} & \text{ if } & v_{i} = b,
     \end{array}
\label{equation_picking_cycle}
\end{equation}
where all the indices are modulo $N$.
Since all the letters $v_{1}$ through $v_{N}$
are not the same, we may assume that
$v_{N-1} = a$ and $v_{N} = b$.
Pick $x_{1}$ arbitrarily. Pick $x_{2}$ through
$x_{N-1}$ such that
inequality~(\ref{equation_picking_cycle}) is satisfied.
Finally, pick $x_{N}$ in the interval
$(\max(x_{N-1},x_{1}),1)$.
Now in the graph $H_{S(U)}$
we have the cycle
$$  (x_{1}, x_{2}, \ldots, x_{m})
           \longrightarrow
    (x_{2}, x_{3}, \ldots, x_{m+1})
           \longrightarrow
        \cdots
           \longrightarrow
    (x_{N}, x_{1}, \ldots, x_{m-1}) 
           \longrightarrow
    (x_{1}, x_{2}, \ldots, x_{m})  . $$
\end{proof}

\begin{theorem}
Let $U$ be a subset of $\{a,b\}^{m}$. If the de Bruijn graph $D_{U}$ is
strongly connected, then the graph $H_{S(U)}$ is also
strongly connected.
Furthermore,
the de Bruijn graph $D_{U}$
has the same period as the graph $H_{S(U)}$.
\label{theorem_D_H}
\end{theorem}
\begin{proof}
Note that the graph $D_{U}$ has
the directed edge $a^{m-1} b$
since otherwise it would not be strongly connected.
Given two vertices $x$ and $y$ in $H_{S(U)}$.
To prove that $H_{S(U)}$ is
strongly connected
it is enough to find a directed path from $x$ to $y$.

We can find a path from $u(\Pi(x))$ to $a^{m-1}$ in
the graph $D_{U}$ that consists of at least $m+1$ edges.
Similarly to the lifting lemma,
Lemma~\ref{lemma_lifting}, we can lift
this path to a path in $H_{S(U)}$
that starts at the vertex $x$ and ends, say,
in the vertex $z = (z_{1}, \ldots, z_{m})$.
Note that $z_{1} < \cdots < z_{m}$.
Moreover, we can find a path from
$a^{m-2} b$ to $u(\Pi(y))$ in $D_{U}$
that has length at least $m+1$.
Lift this path to a path that ends in the vertex $y$
and begins at $w = (w_{1}, \ldots, w_{m})$,
where $w_{1} < \cdots < w_{m-1} > w_{m}$.

Let $v_{i} = \max(z_{i}, w_{i-1})$ for $2 \leq i \leq m$.
Observe that we have the string of inequalities
$z_{1} < v_{2} < \cdots < v_{m} > w_{m}$.
We can now concatenate these two paths as follows.
Replace each occurrence of $z_{i}$ and $w_{i-1}$ by $v_{i}$
for $2 \leq i \leq m$ in each of the two paths.
Then we may connect the vertex
$(z_{1}, v_{2}, \ldots, v_{m})$ with
the vertex
$(v_{2}, \ldots, v_{m},w_{m})$
via the edge that goes across the edge
with descent word $a^{m-1} b$.
Thus the graph $H_{S(U)}$ is strongly connected.

Since there is a graph homomorphism from $H_{S(U)}$
to $D_{U}$ we know that the period of $D_{U}$
divides the period of $H_{S(U)}$.
To see that the periods are equal, pick a vertex~$w$
of $D_{U}$ that differs from $a^{m-1}$ and $b^{m-1}$.
Then any cycle of length greater than $m+1$ through
the vertex $w$
in $D_{U}$
lifts to a cycle of the same length
in $H_{S(U)}$.
Hence the greatest common divisor of 
lengths of cycles through $w$
is a multiple of the greatest common divisor of
the cycle lengths of $H_{S(U)}$.
Hence the two periods are equal.

Note that this argument only works when $m \geq 3$
since there is no such vertex $w$ in the $m=2$ case.
But the remaining $m=2$ case is straightforward to check.
\end{proof}

\subsection{Invariant subspace for descent pattern avoidance}
\label{subsection_invariant_subspace}

For an $ab$-word $u$ of length $m-1$ define the descent polytope $P_{u}$ to
be the subset of the unit cube~$[0,1]^{m}$ corresponding to all vectors with
descent word $u$. That is,
$$
P_{u}
  =
 \{(x_{1}, \ldots, x_{m}) \in[0,1]^{m}
    \:\: : \:\:
       x_{i} \leq x_{i+1} \text{ if } u_{i} = a
          \text{ and }
       x_{i} \geq x_{i+1} \text{ if } u_{i} = b
             \} .
$$
Observe that the $m$-dimensional unit cube is the union of the $2^{m-1}$
descent polytopes $P_{u}$. Now the operator $T$ corresponding to the descent
pattern avoidance of the set $U$ has the following form. For an $ab$-word $u$
of length $m-2$ and $y \in\{a,b\}$ we have
\begin{eqnarray}
\left. T(f) \right| _{P_{u y}}
    & = &
\int_{0}^{x_{1}} \chi(a u y) \cdot
   \left. f(t,x_{1}, \ldots, x_{m-1}) \right|_{P_{a u}}  dt 
\label{equation_descent_y} \\
    & + & 
\int_{x_{1}}^{1} \chi(b u y) \cdot
   \left. f(t,x_{1}, \ldots, x_{m-1}) \right|_{P_{b u}}  dt ,
\nonumber
\end{eqnarray}
where by abuse of notation we let $\chi(w) = 1$ if $w$ does not belong to
the set $U$ and $\chi(w) = 0$ otherwise.

\begin{proposition}
Let $T$ be the operator associated with a descent pattern avoidance and $k$
is an integer such that $0 \leq k \leq m-1$. Let $u$ be an $ab$-word of
length $m-1$. Then the function $T^{k}(f)$ restricted to the descent
polytope $P_{u}$ only depends on the variables $x_{1}$ through $x_{m-k}$.
\end{proposition}
\begin{proof}
Proof by induction on $k$. When $k=0$
there is nothing to prove.
When $1 \leq k \leq m-1$,
we know by induction that the restriction of $T^{k-1}(f)$
only depends on $x_{1}, \ldots, x_{m-k+1}$.
By the shift of variables in the right
hand side of equation~(\ref{equation_descent_y}),
we obtain that $T^{k}(f)$ does
not depend on the variable $x_{m-k+1}$,
completing the induction.
\end{proof}

\begin{corollary}
Let $T$ be the operator associated with a descent pattern avoidance and let
$\varphi$ be an eigenfunction associated with a non-zero eigenvalue $\lambda$.
Then the eigenfunction restricted
to each descent polytope $P_{u}$ only depends on the variable $x_{1}$.
\label{corollary_descent_word_avoidance_eigenfunctions}
\end{corollary}
\begin{proof}
Since $\lambda^{m-1} \cdot \varphi= T^{m-1}(\varphi)$ the eigenfunction has the
required form.
\end{proof}
\begin{corollary}
Let $T$ be the operator associated with a descent pattern avoidance and let
$\varphi$ be an eigenfunction associated with a non-zero eigenvalue $\lambda$.
Assume that $f$ is a generalized eigenfunction, that is,
it satisfies the equation
$\lambda \cdot f = T(f) + \varphi$.
Then the function $f$ restricted
to each descent polytope~$P_{u}$ only depends on the variable $x_{1}$.
\label{corollary_descent_word_avoidance_eigenfunctions_II}
\end{corollary}
\begin{proof}
By induction on $k$. Assume that $1 < k \leq m$
and that $f$ restricted to each descent polytope
only depends on the variables $x_{1}$ through $x_{k}$.
Then $T(f) + \varphi$ only depends on $x_{1}$ through $x_{k-1}$
showing that $\lambda \cdot f$ only depends on $x_{1}, \ldots, x_{k-1}$.
\end{proof}

Let $V$ be the subspace of $L^{2}([0,1]^{m})$ consisting of all functions $f$
that only depend on the variable~$x_{1}$ when restricted to each of the
descent polytopes $P_{u}$. Observe that the subspace $V$ is invariant under
the operator $T$. That is, the operator $T$ restricts to the subspace $V$.
Moreover the constant function~$\mathbf{1}$ belongs to $V$. Hence to
understand the behavior of $T^{n}(\mathbf{1})$
it is enough to study this restricted operator.

In order to describe the subspace $V$ more explicitly define for
an $ab$-word $u$ of length $m-1$
the polynomial $h(u;x_{1})$ as follows:
$$
h(u;x_{1}) 
     =
\int_{(x_{1},x_{2}, \ldots, x_{m}) \in P_{u}}  1 dx_{2} \cdots dx_{m} .
$$
These polynomials were first introduced and studied
in~\cite{Ehrenborg_Levin_Readdy}, with different notation.

Let $p$ be a vector $\left( p_{u}(x_{1}) \right)_{u \in\{a,b\}^{m-1}}$.
That is, the vector $p$ consists of one-variable functions in the variable
$x_{1}$ and is indexed by $ab$-words of length $m-1$.
Consider the function $f$ on $[0,1]^{m}$ defined by
$$
\left. f(x_{1}, \ldots, x_{m}) \right| _{P_{u}} = p_{u}(x_{1})
$$
for all $ab$-words $u$ of length $m-1$. Observe that the function $f$
belongs to $L^{2}([0,1]^{m})$, and hence to the invariant subspace $V$, if
and only if
$$
\int_{0}^{1} h(u; x_{1}) \cdot \left| p_{u}(x_{1})\right|^{2}
 dx_{1} < \infty
$$
for all $ab$-words $u$ of length $m-1$. For two functions $f$ and $g$ in the
subspace~$V$, corresponding to the two vectors $\left( p_{u}(x_{1}) \right)
_{u \in\{a,b\}^{m-1}}$ and $\left( q_{u}(x_{1}) \right) _{u \in
\{a,b\}^{m-1}}$, the inner product is given by
$$
\pair{f}{g}
     =
\sum_{u \in\{a,b\}^{m-1}} \int_{0}^{1}
      h(u;x_{1}) \cdot
      p_{u}(x_{1}) \cdot
      \overline{q_{u}(x_{1})} dx_{1} .
$$
We end this section by a structural result about the subspace $V$.

\begin{proposition}
The invariant subspace $V$ is isometrically isomorphic
to the Hilbert space
$$ L^{2}\left( \left[0,1\right] \right)^{2^{m-1}} . $$
\end{proposition}
\begin{proof}
The isomorphism of the Hilbert spaces
$V \longrightarrow {L^{2}([0,1])}^{2^{m-1}}$ is given by
$$
\left( p_{u}(x_{1})\right) _{u\in\{a,b\}^{m-1}}
  \longmapsto
\left( \sqrt{h(u;x_{1})}\cdot p_{u}(x_{1})\right)_{u\in\{a,b\}^{m-1}}.
$$
\end{proof}

\subsection{Symmetries}

Let $J$ and $R$ be the following two involutions on the space
$L^{2}([0,1]^{m})$:
\begin{eqnarray*}
(J f)(x_{1},x_{2},\ldots,x_{m})
  & = &
f(1-x_{m},\ldots,1-x_{2},1-x_{1}), \\
(R f)(x_{1},x_{2},\ldots,x_{m})
  & = &
f(x_{m},\ldots,x_{2},x_{1}).
\end{eqnarray*}
Observe that both $J$ and $R$ are self adjoint operators.

\begin{lemma}
Assume that $\chi$ has the symmetry
$$
\chi(x_{1},x_{2},\ldots,x_{m},x_{m+1})
   =
\chi(1-x_{m+1},1-x_{m},\ldots,1-x_{2},1-x_{1}).
$$
Then the adjoint of the associated operator $T$ is given by
$T^{*} = J T J$.
Moreover, if $\varphi$ is an eigenfunction of the operator $T$
with eigenvalue $\lambda$ then $\psi = J\varphi$ is an eigenfunction
of the adjoint $T^{*}$ with the eigenvalue~$\lambda$.
Furthermore, we have the equality
$\pair{\mathbf{1}}{\overline{\psi}} = \pair{\varphi}{\mathbf{1}}$.
\label{lemma_J}
\end{lemma}
\begin{proof}
We have that
\begin{eqnarray*}
J T J f(x_{1}, x_{2}, \ldots, x_{m})
  & = &
J T f(1 - x_{m}, \ldots, 1 - x_{2},
1 - x_{1}) \\
  & = &
J \int_{0}^{1} 
    \chi(t,x_{1}, \ldots, x_{m})
      \cdot
    f(1 - x_{m-1}, \ldots, 1 - x_{1}, 1 - t) dt \\
  & = &
\int_{0}^{1}
    \chi(t, 1 - x_{m}, \ldots, 1 - x_{1})
      \cdot
    f(x_{2}, \ldots, x_{m}, 1 - t) dt \\
  & = &
\int_{0}^{1}
    \chi(1 - t, 1 - x_{m}, \ldots, 1 - x_{1})
      \cdot
    f(x_{2}, \ldots, x_{m}, t)  dt \\
  & = &
\int_{0}^{1}
    \chi(x_{1}, \ldots, x_{m},t)
      \cdot
    f(x_{2}, \ldots, x_{m}, t) dt \\
  & = &
T^{*} f(x_{1}, \ldots, x_{m-1}, x_{m}) .
\end{eqnarray*}
For the second statement consider the following line of equalities
$T^{*} J \varphi= J T J J \varphi= J T \varphi= \lambda \cdot J \varphi$.
Lastly,
$\pair{\mathbf{1}}{\overline{\psi}}
   =
 \pair{\mathbf{1}}{\overline{J \varphi}}
   =
 \pair{\mathbf{1}}{J \overline{\varphi}}
   =
 \pair{J \mathbf{1}}{\overline{\varphi}}
   =
 \pair{\mathbf{1}}{\overline{\varphi}}
   =
 \pair{\varphi}{\mathbf{1}}$.
\end{proof}

Similarly to Lemma~\ref{lemma_J} we have the next lemma. Its proof is
similar to the previous proof and hence omitted.

\begin{lemma}
Assume that $\chi$ has the symmetry
$$
\chi(x_{1},x_{2},\ldots,x_{m},x_{m+1})
=
\chi(x_{m+1},x_{m},\ldots,x_{2},x_{1}).
$$
Then we have that the adjoint of the associated operator $T$ is given by
$T^{*} = R T R$.
Moreover, if $\varphi$ is an eigenfunction of the operator $T$
with eigenvalue $\lambda$ then $\psi = R\varphi$
is an eigenfunction of the adjoint~$T^{*}$
with the eigenvalue~$\lambda$.
Furthermore, we have the equality
$\pair{\mathbf{1}}{\overline{\psi}} = \pair{\varphi}{\mathbf{1}}$.
\label{lemma_R}
\end{lemma}

\section{123-Avoiding permutations}
\label{section_123}

A 123-avoiding permutation is a permutation $\pi\in\mathfrak{S}_{n}$ with no
index $j$ so that $\pi_{j}<\pi_{j+1}<\pi_{j+2}$, where $1\leq j\leq n-2$.
Let $\alpha_{n}(123)$ denote the number of $123$-avoiding permutations
in $\mathfrak{S}_{n}$.

\subsection{Eigenvalues and eigenfunctions}

Since $123$-avoiding permutations can be viewed as permutations with no
double descents
Corollary~\ref{corollary_descent_word_avoidance_eigenfunctions}
allows us to recast the
problem of finding eigenfunctions in two variables into finding two
one-variable functions.

\begin{proposition}
The eigenvalues of the operator $T$ are given by
\begin{equation}
\lambda_{k}=\frac{\sqrt{3}}{2 \cdot \pi \cdot \left( k+\frac{1}{3}\right) },
\label{equation_123_eigenvalue}
\end{equation}
where $k\in\mathbb{Z}$ and the associated eigenfunctions
are given by
\begin{equation}
\varphi_{k}=\exp\left( -\frac{x}{2 \cdot \lambda}\right) \cdot
\left\{
\begin{array}{c l}
{\cos\left( \frac{\pi}{6}+\frac{\sqrt{3}}{2}\cdot \frac{x}{\lambda}\right) }
& \text{ if }0 \leq x \leq y \leq 1, \\[2 mm]
{\sin\left( \frac{\pi}{3}+\frac{\sqrt{3}}{2}\cdot \frac{x}{\lambda}\right) }
& \text{ if }0 \leq y \leq x \leq 1.
\end{array}
\right.
\label{equation_123_phi_k}
\end{equation}
\label{proposition_123_eigenvalue}
\end{proposition}
\begin{proof}
Avoiding the pattern $123$ is equivalent to
avoiding the descent set pattern $aa$.
Hence Corollary~\ref{corollary_descent_word_avoidance_eigenfunctions}.
states that the eigenfunctions $\varphi$
can be written as
$$\varphi = \left\{
              \begin{array}{c l}
                 p(x) & \text{ if } 0 \leq x \leq y \leq 1 , \\
                 q(x) & \text{ if } 0 \leq y \leq x \leq 1 .
              \end{array}
            \right. $$
Then the defining equations
for eigenvalues and eigenfunctions reduces to the
integral system:
\begin{eqnarray}
\lambda \cdot p(x) & = & \int_{x}^{1} q(t) dt ,
\label{equation_123_integral.system.1} \\
\lambda \cdot q(x) & = & \int_{0}^{x} p(t) dt+\int_{x}^{1} q(t) dt .
\label{equation_123_integral.system.2}
\end{eqnarray}
First, differentiating with respect to $x$, we obtain the first-order system
\begin{eqnarray}
\lambda \cdot p^{\prime}(x) & = & -q(x) ,
\label{equation_d_p} \\
\lambda \cdot q^{\prime}(x) & = & p(x)-q(x) .
\label{equation_d_q}
\end{eqnarray}
These equations have only the trivial solution if $\lambda=0$, so $\lambda=0$
is not an eigenvalue. If $\lambda\neq0$ then
the first-order system~(\ref{equation_d_p})--(\ref{equation_d_q})
implies the second-order equation
$$
\lambda^{2}\cdot p^{\prime\prime}(x)+\lambda \cdot p^{\prime}(x)+p(x)=0.
$$
This equation has the general solution
\begin{equation}
p(x)=A\cdot \exp\left( \frac{\omega}{\lambda}\cdot x\right) +B\cdot
\exp\left( \frac{\omega^{2}}{\lambda}\cdot x\right) ,
\label{equation_p_gen_123}
\end{equation}
where $\omega=\exp\left( \frac{2 \cdot \pi \cdot i}{3}\right) $. That is, $\omega$
satisfies the relation $\omega^{2}+\omega+1=0$.
Moreover, equation~(\ref{equation_d_p}) implies that
\begin{equation}
q(x) 
     =
- \omega \cdot A
         \cdot      \exp\left( \frac{\omega}{\lambda}\cdot x\right)
- \omega^{2} \cdot B 
             \cdot \exp\left( \frac{\omega^{2}}{\lambda}\cdot x\right) .
\label{equation_q_gen_123}
\end{equation}

Setting $x=0$ and $x=1$ in equations~(\ref{equation_123_integral.system.1})
and~(\ref{equation_123_integral.system.2})
and using that $\lambda\neq0$ we obtain the
boundary conditions:
\begin{eqnarray}
p(0) & = & q(0) ,
\label{equation_boundary_condition_0_123} \\
p(1) & = & 0 .
\label{equation_boundary_condition_1_123}
\end{eqnarray}
Substituting the expressions for $p(x)$ and $q(x)$ from
equations~(\ref{equation_d_p}) and~(\ref{equation_d_q})
into boundary
condition~(\ref{equation_boundary_condition_0_123})
we
obtain $A+B=-\omega\cdot A-\omega^{2}\cdot B$. This is equivalent to
$\omega\cdot A+B=0$. Hence we may set $A=1/2 \cdot \exp\left( \frac{\pi \cdot i}{6}
\right)$ and $B=\overline{A}=1/2 \cdot \exp\left( -\frac{\pi \cdot i}{6}\right) $.
Substituting equation~(\ref{equation_d_p}) into
the second boundary condition~(\ref{equation_boundary_condition_1_123})
implies that
$$
A\cdot \exp\left( \frac{\omega}{\lambda}\right)
=
-B\cdot \exp\left( \frac{\omega^{2}}{\lambda}\right)
=
\omega\cdot A\cdot \exp\left( \frac {\omega^{2}}{\lambda}\right) .
$$
Cancelling $A$ on both sides and taking the logarithm gives
$$
\frac{\omega}{\lambda}
=
\frac{2 \cdot \pi \cdot i}{3}+\frac{\omega^{2}}{\lambda}+2 \cdot \pi \cdot i\cdot k,
$$
where $k$ is an integer. Since $\omega-\omega^{2}=\sqrt{3}\cdot i$ we obtain
expression~(\ref{equation_123_eigenvalue}). Moreover $p(x)$ is given by
\begin{eqnarray*}
p(x)
  & = &
\frac{1}{2} \cdot
\exp\left( \frac{\pi \cdot i}{6}\right) \cdot
\exp\left( \omega \cdot \frac{x}{\lambda}\right)
+
\frac{1}{2} \cdot
\exp\left( -\frac{\pi \cdot i}{6}\right) \cdot
\exp\left( \omega^{2} \cdot \frac{x}{\lambda}\right) \\
  & = &
\frac{\exp\left( -\frac{x}{2 \cdot \lambda}\right) }{2}
\cdot
\left(
   \exp\left(\frac{\pi \cdot i}{6}
         +
             \frac{\sqrt{3}}{2}\cdot i\cdot \frac{x}{\lambda}\right)
  +
   \exp\left( -\frac{\pi \cdot i}{6}
         -
             \frac{\sqrt{3}}{2}\cdot i\cdot \frac{x}{\lambda}\right)
\right) \\
  & = &
\exp\left( -\frac{x}{2 \cdot \lambda}\right)
   \cdot
\cos\left( \frac{\pi}{6}
             +
           \frac{\sqrt{3}}{2} \cdot \frac{x}{\lambda}\right) .
\end{eqnarray*}
Now equation~(\ref{equation_d_p}) implies the claimed expression for $q(x)$.
\end{proof}

Note that the eigenvalues are ordered by
$$
\lambda_{0}>-\lambda_{-1}>\lambda_{1}>-\lambda_{-2}>\lambda_{2}>-\lambda
_{-3}>\lambda_{3}>\cdots>0 .
$$
Furthermore, the calculations in 
Proposition~\ref{proposition_123_eigenvalue}
showed that they all have a unique eigenfunction.
It remains to show that they are simple.

\begin{proposition}
The eigenvalues $\lambda$ of the operator $T$ are simple, that is,
they have index $1$.
In other words, 
there is no function $f(x,y)$ such that
$\lambda_{k} \cdot f = T(f) + \varphi_{k}$.
\label{proposition_123_simple}
\end{proposition}
\noindent
{\em Sketch of proof.}
Using
Corollary~\ref{corollary_descent_word_avoidance_eigenfunctions_II}
we can write the function $f$ as
$$ f(x,y) = \left\{\begin{array}{c l}
                     r(x) & \text{ if } 0 \leq x \leq y \leq 1, \\
                     s(x) & \text{ if } 0 \leq y \leq x \leq 1.
                   \end{array} \right.  $$
Then equation reduces to the
integral system:
\begin{eqnarray*}
\lambda \cdot r(x)
  & = &
\int_{x}^{1} s(t) dt +  p(x) , \\
\lambda \cdot s(x)
  & = &
\int_{0}^{x} r(t) dt + \int_{x}^{1} s(t) dt + q(x) .
\end{eqnarray*}
Note that we obtain the two boundary conditions
$r(1) = 0$ and $r(0) = s(0)$.
Next differentiating with respect to $x$, we have the first-order system:
\begin{eqnarray*}
\lambda \cdot r^{\prime}(x) & = & -s(x)     + p^{\prime}(x) , \\
\lambda \cdot s^{\prime}(x) & = & r(x)-s(x) + q^{\prime}(x) .
\end{eqnarray*}
This system of differential equations
can be solved as in the proof
of Proposition~\ref{proposition_123_eigenvalue}.
However, the solution does not satisfy the boundary conditions,
completing the sketch.

By applying the involution $J$ we obtain the adjoint eigenfunction
\begin{equation}
\psi_{k}=\exp\left( \frac{y-1}{2 \cdot \lambda}\right) \cdot \left\{
\begin{array}{c l}
\cos\left( \frac{\pi}{6}+\frac{\sqrt{3}}{2}\cdot \frac{1-y}{\lambda}\right) &
\text{ if } 0 \leq x \leq y \leq 1, \\[2 mm]
\sin\left( \frac{\pi}{3}+\frac{\sqrt{3}}{2}\cdot \frac{1-y}{\lambda}\right) &
\text{ if } 0 \leq y \leq x \leq 1.
\end{array}
\right.
\label{equation_123_psi_k}
\end{equation}

\begin{proposition}
Let $\lambda$ be an eigenvalue of $T$
with eigenfunctions $\varphi$ and
let $\psi$ be the eigenfunction
of the adjoint operator $T^{*}$ with eigenvalue $\lambda$.
Then the following identities hold:
\begin{eqnarray}
\pair{\varphi}{\mathbf{1}}
  & = &
\pair{\mathbf{1}}{\overline{\psi}}
   =
\frac{\sqrt{3}}{2} \cdot \lambda^{2} ,
\label{equation_123_1_phi} \\
\pair{\varphi}{\overline{\psi}}
  & = &
\frac{3}{4} \cdot (-1)^{k} \cdot
\lambda \cdot \exp\left( -\frac {1}{2 \cdot \lambda}\right) .
\label{equation_123_psi_phi}
\end{eqnarray}
In particular
\begin{equation}
\frac{\pair{\varphi}{\mathbf{1}}
      \cdot \pair{\mathbf{1}}{\overline{\psi}}}
     {\pair{\varphi}{\overline{\psi}}}
 =
 (-1)^{k} \cdot \lambda^{3} \cdot
\exp\left( \frac{1}{2 \cdot \lambda}\right) .
\label{equation_123_res}
\end{equation}
\label{proposition_123_res}
\end{proposition}
\begin{proof}
In the following calculations we use the facts that
$\cos(\sqrt {3}/(2 \cdot \lambda))=(-1)^{k}/2$
and
$\sin(\sqrt{3}/(2 \cdot \lambda))=(-1)^{k} \sqrt{3}/2$.
We also use the expression for $\varphi$
in equation~(\ref{equation_123_phi_k}).
First, we note that
\begin{eqnarray*}
\pair{\varphi}{\mathbf{1}}
  & = &
    \int_{0 \leq x \leq y \leq 1} p(x) dx dy
  + \int_{0 \leq y \leq x \leq 1} q(x) dx dy \\
  & = &
\int_{0}^{1} (1-x) p(x) \cdot dx + \int_{0}^{1} x \cdot q(x) dx .
\end{eqnarray*}
Explicit computation shows that
\begin{eqnarray*}
\int_{0}^{1} (1-x) \cdot p(x) dx
  & = &
\frac{\sqrt{3}}{2} \cdot \lambda^{2} \cdot
  (1 - (-1)^{k} \cdot \exp(-1/(2 \cdot \lambda))) \\
\int_{0}^{1} x \cdot q(x) dx 
  & = &
\frac{\sqrt{3}}{2} \cdot \lambda^{2} \cdot(-1)^{k}
\cdot \exp(-1/(2 \cdot \lambda))
\end{eqnarray*}
which shows~(\ref{equation_123_1_phi}). Next, using~(\ref{equation_123_phi_k})
and~(\ref{equation_123_psi_k}) and dropping subscripts as before, we have
\begin{eqnarray*}
\pair{\varphi}{\overline{\psi}}
  & = &
  \int_{0 \leq x \leq y \leq 1} p(x) \cdot p(1-y) dx dy
+ \int_{0 \leq y \leq x \leq 1} q(x) \cdot q(1-y) dx dy \\
  & = &
\int_{0}^{1} \left( p(x)\cdot \int_{x}^{1} p(1-y) dy
+q(x)\cdot \int_{0}^{x} q(1-y) dy\right) dx.
\end{eqnarray*}
Carrying out the $y$ integration and simplifying, we obtain
$$
\pair{\varphi}{\overline{\psi}}
  =
\frac{3}{4} \cdot (-1)^{k} \cdot
\int_{0}^{1} \lambda \cdot \exp\left( -\frac {1}{2 \cdot \lambda}\right) dx
$$
which gives~(\ref{equation_123_psi_phi}).
\end{proof}

\subsection{Asymptotics}

The above computations show that all eigenvalues of $T$ are simple and
give the eigenvalues and the coefficients explicitly.
We thus obtain the following expansion for
$\pair{T^{n-2}(\mathbf{1})}{\mathbf{1}}
   =
 \alpha_{n}(123)/n!$
as an immediate consequence of 
Theorem~\ref{theorem_expansion},
Propositions~\ref{proposition_123_eigenvalue}
and~\ref{proposition_123_res}.

\begin{theorem}
\label{theorem_123}
Let $K$ be a non-negative integer.
The number of $123$-avoiding permutations
satisfies the following asymptotic expansion
$$
\frac{\alpha_{n}(123)}{n!}
   =
\sum_{\left| k\right| \leq K}
    (-1)^{k}
      \cdot
    \exp\left( \frac{1}{2 \cdot \lambda_{k}}\right)
      \cdot
    \lambda_{k}^{n+1}
   +
 O\left( r_{K+1}^{n}\right)  ,
$$
where $\lambda_{k}$ is given by~(\ref{equation_123_eigenvalue}),
$r_{k} = | \lambda_{-k} |
       = \sqrt{3}/\left(2 \cdot \pi \cdot \left( k-\frac{1}{3}\right)\right)$
and the sum contains $2K+1$ terms
corresponding to the $2K+1$ largest eigenvalues.
\end{theorem}

\section{213-Avoiding permutations}
\label{section_213}

A 213-avoiding permutation is a permutation $\pi\in\mathfrak{S}_{n}$ which
contains no sequence of the form
$$
\pi_{j+1}<\pi_{j}<\pi_{j+2}
$$
for any $j$ with $1\leq j\leq n-2$. We denote the number of
213-avoiding permutations of $\mathfrak{S}_{n}$ by
$\alpha_{n}(213)$. Thus, $S$ consists of the single permutation
$213$ and
$$
\chi(x_{1},x_{2},x_{3})
=
\left\{ \begin{array}{c l}
           0 & \text{if } x_{2} \leq x_{1} \leq x_{3}, \\
           1 & \text{otherwise}.
        \end{array} \right. $$

By symmetry, the study of $213$-avoiding permutations
is equivalent to $132$-avoiding permutations,
$231$-avoiding permutations and $312$-avoiding
permutations. However the case of $213$-avoiding permutations gives the most
straightforward equations.

\subsection{Eigenvalues and eigenfunctions}

In what follows, we will make use of the error function
\begin{equation}
\erf(x)
   =
\frac{2}{\sqrt{\pi}} \cdot \int_{0}^{x}\exp(-t^{2}) dt
\label{equation_erf}
\end{equation}
which extends to an entire function on $\mathbb{C}$, and the function
\begin{equation}
q(x)=\exp\left( -\frac{x^{2}}{2 \cdot \lambda^{2}}\right) .
\label{equation_q}
\end{equation}

Let
$$ f(x,y)
   =
\left\{ \begin{array}{c l}
           p(x,y) & \text{ if }0 \leq x \leq y \leq 1, \\
           q(x,y) & \text{ if }0 \leq y \leq x \leq 1.
        \end{array} \right. $$
Then
$$
(T f)(x,y)
  =  
\left\{ \begin{array}{c l}
            \int_{0}^{x} p(t,x) dt + \int_{y}^{1} q(t,x) dt &
            \text{if }0 \leq x \leq y \leq 1, \\[2 mm]
            \int_{0}^{x} p(t,x) dt + \int_{x}^{1} q(t,x) dt &
            \text{if }0 \leq y \leq x \leq 1.
        \end{array} \right. $$
Now we characterize the nonzero eigenvalues and eigenfunctions.

\begin{proposition}
The non-zero eigenvalues $\lambda$ of the operator $T$ satisfy the
equation
\begin{equation}
\erf\left( \frac{1}{\sqrt{2}\cdot \lambda}\right)
   =
\sqrt{\frac{2}{\pi}}
\label{equation_213_eigenvalue}
\end{equation}
and the corresponding eigenfunctions are
$$
\varphi(x,y)
  =
\left\{ \begin{array}{c l}
           q(x) - \frac{1}{\lambda} \cdot \int_{x}^{y} q(t) dt &
           \text{if } 0 \leq x \leq y \leq 1, \\[2 mm]
           q(x) &
           \text{if } 0 \leq y \leq x \leq 1,
        \end{array} \right. $$
where $q(x)$ is given by~(\ref{equation_q}).
\end{proposition}
\begin{proof}
The defining relations for the eigenfunctions are
\begin{eqnarray}
\lambda \cdot p(x,y)
   & = &
\int_{0}^{x} p(t,x) dt + \int_{y}^{1} q(t,x) dt ,
\label{equation_213_defining_I} \\
\lambda \cdot q(x,y)
   & = &
\int_{0}^{x} p(t,x) dt + \int_{x}^{1} q(t,x) dt .
\label{equation_213_defining_II}
\end{eqnarray}
Now observe that in the right-hand side of
equation~(\ref{equation_213_defining_II})
there is no dependency on the variable $y$.
Hence we may replace $q(x,y)$ with $q(x)$. Now subtract
equation~(\ref{equation_213_defining_II})
from equation~(\ref{equation_213_defining_I})
$$
\lambda \cdot(p(x,y)-q(x))=-\int_{x}^{y} q(t) dt.
$$
That is,
\begin{equation}
p(x,y)=q(x)-\frac{1}{\lambda} \cdot\int_{x}^{y} q(t) dt.
\label{equation_213_p}
\end{equation}
Substitute equation~(\ref{equation_213_p}) into
equation~(\ref{equation_213_defining_II}):
\begin{eqnarray*}
\lambda \cdot q(x) 
  & = &
\int_{0}^{x}\left( q(t)-\frac{1}{\lambda} \cdot\int_{t}^{x} q(s) ds\right) dt
    +
\int_{x}^{1} q(t) dt \\
  & = &
\int_{0}^{1} q(t) dt-\frac{1}{\lambda} \cdot\int_{0}^{x}\int_{t}^{x} q(s) ds dt \\
  & = &
\int_{0}^{1} q(t) dt-\frac{1}{\lambda} \cdot\int_{0}^{x}\int_{0}^{s} q(s) dt ds \\
  & = &
\int_{0}^{1} q(t) dt-\frac{1}{\lambda} \cdot\int_{0}^{x} s \cdot q(s) ds.
\end{eqnarray*}
Hence we have the following integral equation for $q(x)$
\begin{equation}
\lambda^{2}\cdot q(x)
   =
\lambda \cdot \int_{0}^{1} q(t) dt-\int_{0}^{x} s\cdot q(s) ds.
\label{equation_213_integral_equation}
\end{equation}
Differentiating once we have
\begin{equation}
\lambda^{2}\cdot q^{\prime}(x)=-x\cdot q(x).
\label{equation_213_differential_equation}
\end{equation}
The solution to this differential equation is
\begin{equation}
q(x)=C\cdot \exp\left( -\frac{x^{2}}{2\cdot \lambda^{2}}\right) .
\label{equation_213_solution_differential_equation}
\end{equation}
By setting the constant $C$ to be $1$ we obtain a solution to
equation~(\ref{equation_213_differential_equation}).
Now substitute this solution for $q(x)$
into the integral equation~(\ref{equation_213_integral_equation}) and set
$x=0$:
\begin{eqnarray*}
\lambda^{2}
  & = &
\lambda \cdot
  \int_{0}^{1} \exp\left( -\frac{t^{2}}{2\cdot \lambda^{2}}\right) dt \\
  & = &
\sqrt{2} \cdot \lambda^{2} \cdot
\int_{0}^{1/(\sqrt{2} \cdot \lambda)} \exp\left(-u^{2}\right) du \\
  & = &
\frac{\sqrt{\pi}\cdot \lambda^{2}}{\sqrt{2}}
   \cdot
\erf(1/(\sqrt{2} \cdot \lambda)),
\end{eqnarray*}
where the substitution in the integral is
$u=t/(\sqrt{2} \cdot \lambda)$.
Hence the non-zero eigenvalues $\lambda$
satisfy equation~(\ref{equation_213_eigenvalue}).
\end{proof}

\begin{proposition}
All the non-zero eigenvalues of the operator $T$
are simple, that is, there are no
non-zero eigenvalue whose index is greater than or equal to $2$.
\label{proposition_213_simple}
\end{proposition}
\begin{proof}
Assume that $\lambda$ is a non-zero eigenvalue
of index at least $2$,
that is, there is a function $f(x,y)$ such that
$\lambda \cdot f = T(f) + \varphi$.
Letting
$f(x,y) = r(x,y)$ for $0 \leq x \leq y \leq 1$
and 
$f(x,y) = s(x,y)$ for $0 \leq y \leq x \leq 1$
we have 
\begin{eqnarray}
\lambda \cdot r(x,y)
  & = &
\int_{0}^{x} r(t,x) dt + \int_{y}^{1} s(t,x) dt + p(x,y) ,
\label{equation_213_simple_I} \\
\lambda \cdot s(x,y)
  & = &
\int_{0}^{x} r(t,x) dt + \int_{x}^{1} s(t,x) dt + q(x) .
\label{equation_213_simple_II}
\end{eqnarray}
Note that in the right-hand side of
equation~(\ref{equation_213_simple_II})
there is no dependency on the variable $y$.
Hence $s(x,y)$ is a function of $x$ only,
and we write $s(x)$ henceforth.
Subtracting 
equation~(\ref{equation_213_simple_II})
from
equation~(\ref{equation_213_simple_I})
and dividing by $\lambda$ gives
\begin{eqnarray*}
r(x,y)
  & = &
s(x) 
- \frac{1}{\lambda} \cdot \int_{x}^{y} s(t) dt
+ \frac{1}{\lambda} \cdot (p(x,y) - q(x)) \\
  & = &
s(x) 
- \frac{1}{\lambda^{2}} \cdot \int_{x}^{y} (\lambda \cdot s(t) + q(t)) dt,
\end{eqnarray*}
where the last step is by
equation~(\ref{equation_213_p}).
Substituting this expression into
equation~(\ref{equation_213_simple_II}) yields
the integral equation for $s(x)$:
\begin{eqnarray}
\lambda \cdot s(x)
  & = &
\int_{0}^{x} \left(
s(t) 
- \frac{1}{\lambda^{2}} \cdot \int_{t}^{x} (\lambda \cdot s(u) + q(u)) du
\right) dt
 + \int_{x}^{1} s(t) dt + q(x)  \nonumber \\
  & = &
- \frac{1}{\lambda^{2}} \cdot 
\int_{0}^{x} \int_{t}^{x} (\lambda \cdot s(u) + q(u)) du dt
 + \int_{0}^{1} s(t) dt + q(x)  \nonumber \\
  & = &
- \frac{1}{\lambda^{2}} \cdot 
\int_{0}^{x} \int_{0}^{u} (\lambda \cdot s(u) + q(u)) dt du
 + \int_{0}^{1} s(t) dt + q(x)  \nonumber \\
  & = &
- \frac{1}{\lambda^{2}} \cdot 
\int_{0}^{x} u \cdot (\lambda \cdot s(u) + q(u)) du
 + \int_{0}^{1} s(t) dt + q(x)  \nonumber \\
  & = &
- \frac{1}{\lambda} \cdot 
\int_{0}^{x} u \cdot s(u) du
 + \int_{0}^{1} s(t) dt 
 - \frac{1}{\lambda} \cdot \int_{0}^{1} q(t) dt 
+ 2 \cdot q(x) ,
\label{equation_213_simple_integral_equation}
\end{eqnarray}
where the last step is using
equation~(\ref{equation_213_integral_equation}).
Differentiate with respect to $x$ to obtain
the first order differential equation in $s(x)$:
$$
\lambda \cdot s^{\prime}(x)
   = 
 - \frac{1}{\lambda} \cdot x \cdot s(x) 
 + 2 \cdot q^{\prime}(x)  .
$$
The general solution to this differential equation is
$$
s(x)
     =
- \frac{x^{2}}{\lambda^{3}} 
    \cdot
  \exp\left(-\frac{x^{2}}{2 \cdot \lambda^{2}}\right)
+
  C \cdot q(x)  , $$
where we used that
$q(x) = \exp\left(-\frac{x^{2}}{2 \cdot \lambda^{2}}\right)$.
Observe that the homogeneous part of this solution
is expected. It corresponds to the homogeneous
part of the equation $\lambda \cdot f = T(f) + \varphi$.
Hence we may set $C = 0$ without loss of generality.
Setting $x = 0$ in~(\ref{equation_213_simple_integral_equation})
yields
$$
   0
  =
   \int_{0}^{1} s(t) dt 
 - \frac{1}{\lambda} \cdot \int_{0}^{1} q(t) dt 
+ 2 \cdot q(0) .
$$
Using that $q(0) = 1$,
$\int_{0}^{1} q(t) dt = \lambda$
and 
\begin{eqnarray*}
   \int_{0}^{1} s(t) dt
  & = &
     \frac{1}{\lambda}
        \cdot
     \exp\left(-\frac{1}{2 \cdot \lambda^{2}}\right)
  -
     \sqrt{\frac{\pi}{2}}
        \cdot
     \erf\left(\frac{1}{\sqrt{2} \cdot \lambda}\right)  \\
  & = &
     \frac{1}{\lambda}
        \cdot
     \exp\left(-\frac{1}{2 \cdot \lambda^{2}}\right)
  -
     1 ,
\end{eqnarray*}
we obtain the equation
$$  0
      =
   \frac{1}{\lambda}
     \cdot 
   \exp\left(-\frac{1}{2 \cdot \lambda^{2}}\right)  $$
which has no solutions. Hence there is no non-zero
eigenvalue of index greater than $2$, that is, all
non-zero eigenvalues are simple.
\end{proof}

For completeness we state:

\begin{lemma}
A function in the kernel of the operator $T$ has the form
$$
\varphi(x,y)
 =
\left\{ \begin{array}{c l}
           p(x,y) & \text{if }0 \leq x \leq y \leq 1, \\
           0      & \text{if }0 \leq y \leq x \leq 1,
        \end{array} \right. $$
where $p(x,y)$ satisfies $\int_{0}^{x} p(t,x) dt = 0$.
\end{lemma}

The adjoint operator $T^{*}$ is given by
$$ T^{*}(f(x,y))
      =
   \left\{ \begin{array}{c l}
             \int_{0}^{y} q(y,u) du+\int_{y}^{1} p(y,u) du &
             \text{if }0 \leq x \leq y \leq 1, \\[2 mm]
             \int_{0}^{y} q(y,u) du+\int_{y}^{x} p(y,u) du &
             \text{if }0 \leq y \leq x \leq 1.
           \end{array} \right. $$

\begin{proposition}
For a non-zero eigenvalues $\lambda$ of the operator $T$ the corresponding
eigenfunction of the adjoint operator $T^{*}$ is
$$ \psi(x,y)
     =
\left\{ \begin{array}{c l}
           p^{*}(y) &
           \text{if} 0 \leq x \leq y \leq 1, \\[2 mm]
           p^{*}(y)-\frac{1}{\lambda}\cdot \int_{x}^{1} p^{*}(u) du &
           \text{if} 0 \leq y \leq x \leq 1,
        \end{array} \right. $$
where
\begin{equation}
p^{*}(y)=-2\cdot y \cdot \exp\left( \frac{y^{2}}{2 \cdot \lambda^{2}}\right)
+2\cdot \lambda +\sqrt{2 \cdot \pi}\cdot y \cdot \exp\left( \frac{y^{2}}{2 \cdot \lambda^{2}}
\right) \cdot \erf\left( \frac{y}{\sqrt{2}\lambda}\right) .
\label{equation_213_adjoint_p}
\end{equation}
\end{proposition}

\begin{proof}
The defining relations for the eigenfunctions are
\begin{eqnarray}
\lambda \cdot p^{*}(x,y)
  & = &
\int_{0}^{y} q^{*}(y,u) du + \int_{y}^{1} p^{*} (y,u) du ,
\label{equation_213_adjoint_defining_I} \\
\lambda \cdot q^{*}(x,y)
  & = &
\int_{0}^{y} q^{*}(y,u) du + \int_{y}^{x} p^{*}(y,u) du .
\label{equation_213_adjoint_defining_II}
\end{eqnarray}
Observe that there is no dependency on the variable $x$ in
equation~(\ref{equation_213_adjoint_defining_I}).
Thus we write $p^{*}(x,y)=p^{*}(y)$.
Subtracting these two equations we have
$$
\lambda \cdot(q^{*}(x,y)-p^{*}(y))=-\int_{x}^{1} p^{*}(u) du,
$$
such that
\begin{equation}
q^{*}(x,y)=p^{*}(y)-\frac{1}{\lambda}\cdot \int_{x}^{1} p^{*}(u) du.
\label{equation_213_adjoint_q}
\end{equation}
Substituting this expression into
equation~(\ref{equation_213_adjoint_defining_I}) one obtains
\begin{eqnarray*}
\lambda \cdot p^{*}(y) 
  & = &
     \int_{0}^{y} 
         \left( p^{*}(u)
               -
              \frac{1}{\lambda} \cdot
              \int_{y}^{1} p^{*}(v) dv
         \right) du
   + \int_{y}^{1} p^{*}(u) du \\
  & = &
\int_{0}^{1} p^{*}(u) du-\frac{1}{\lambda}\cdot \int_{0}^{y}
\int_{y}^{1} p^{*}(v) dv du \\
  & =&
\int_{0}^{1} p^{*}(u) du
 - \frac{1}{\lambda} \cdot y \cdot \int_{y}^{1} p^{*}(v) dv.
\end{eqnarray*}
That is, $p^{*}(y)$ satisfies the integral equation
\begin{equation}
\lambda^{2}\cdot p^{*}(y)=\lambda \cdot \int_{0}^{1} p^{*}(u) du-y \cdot
\int_{y}^{1} p^{*}(v) dv.
\label{equation_213_adjoint_integral_equation}
\end{equation}
Differentiating this equation twice we obtain
\begin{eqnarray}
\lambda^{2} \cdot {p^{*}}^{\prime}(y)
  & = &
y \cdot p^{*}(y)-\int_{y}^{1} p^{*}(v) dv,
\label{equation_213_between_equation} \\
\lambda^{2} \cdot {p^{*}}^{\prime\prime}(y)
  & = &
y \cdot {p^{*}}^{\prime}(y) + 2 \cdot p^{*}(y).
\label{equation_213_adjoint_differential_equation}
\end{eqnarray}
The solution of this differential equation is given by
\begin{equation}
p^{*}(y)
  =
C_{1} \cdot y \cdot \exp\left( \frac{y^{2}}{2 \cdot \lambda^{2}}\right)
+
C_{2} \cdot
\left[ 2 \cdot \lambda +
       \sqrt{2 \cdot \pi} \cdot y \cdot
       \exp\left( \frac{y^{2}}{2 \cdot \lambda^{2}}\right) \cdot
       \erf\left( \frac {y}{\sqrt{2} \cdot \lambda} \right)
\right] .
\label{equation_213_adjoint_solution_differential_equation}
\end{equation}
Setting $y=0$ in equations~(\ref{equation_213_adjoint_integral_equation})
and~(\ref{equation_213_between_equation}) we obtain
$\lambda \cdot p^{*}(0)
    =
 \int_{0}^{1} p^{*}(u) du
    =
 -\lambda^{2} \cdot p^{\prime}(0)$.
Inserting
this condition into the solution of the differential
equation~(\ref{equation_213_adjoint_solution_differential_equation})
we obtain $C_{1}=-2\cdot C_{2}$.
Moreover setting $C_{2}=1$ we obtain equation~(\ref{equation_213_adjoint_p}).
\end{proof}

\begin{lemma}
A function in the kernel of the adjoint operator $T^{*}$ has the form
$$ \psi(x,y)
  =
   \left\{ \begin{array}{c l}
             0          & \text{if }0 \leq x \leq y \leq 1, \\
             q^{*}(x,y) & \text{if }0 \leq y \leq x \leq 1
           \end{array} \right. $$
where $q^{*}(x,y)$ satisfies $\int_{0}^{y} q^{*}(y,u) du = 0$.
\end{lemma}

\begin{proposition}
For a non-zero eigenvalue $\lambda$ with eigenfunction $\varphi$ and 
eigenfunction $\psi$ of the adjoint operator $T^{*}$, we have
\begin{eqnarray*}
\pair{\varphi}{\mathbf{1}}          & = & \lambda^{2} , \\
\pair{\mathbf{1}}{\overline{\psi}}  & = & 2 \cdot \lambda^{3} , \\
\pair{\varphi}{\overline{\psi}}     & = &
         2 \cdot \lambda^{2} \cdot \exp(- 1 / (2 \cdot \lambda^{2})) .
\end{eqnarray*}
In particular,
$$
\frac{\pair{\varphi}{\mathbf{1}} \cdot
      \pair{\mathbf{1}}{\overline{\psi}}}
     {\pair{\varphi}{\overline{\psi}}}
=
  \lambda^{3} \cdot \exp(1 / (2 \cdot \lambda^{2})) .
$$
\label{proposition_213_calculations}
\end{proposition}
\begin{proof}
In the calculations that follows we will use the relations
$\erf(1/(\sqrt{2} \cdot \lambda)) = \sqrt{2/\pi}$,
$\int_{0}^{1} q(x) dx = \lambda$
and $\int_{0}^{1} p^{*}(y) dy = 2 \cdot \lambda^{2}$.

First the inner product between the eigenfunction
and the constant function $\mathbf{1}$:
\begin{eqnarray*}
\pair{\varphi}{\mathbf{1}}
  & = &
\int_{[0,1]^{2}} q(x) dx dy -
\frac{1}{\lambda} \cdot
\int_{0 \leq x \leq y \leq 1} \int_{x}^{y} q(t) dt dx dy \\
  & = &
\int_{0}^{1} q(x) dx - \frac{1}{\lambda} \cdot \int_{0}^{1} t \cdot(1-t) \cdot
q(t) dt \\
  & = &
\lambda- \lambda + \frac{\sqrt{\pi}}{\sqrt{2}} \cdot \lambda^{2} \cdot
\erf(1/(\sqrt{2} \cdot \lambda)) \\
  & = &
\lambda^{2} .
\end{eqnarray*}
Second, the inner product between the adjoint eigenfunction and the constant
function $\mathbf{1}$. We have
\begin{eqnarray*}
\pair{\mathbf{1}}{\overline{\psi}}
   & = & 
\int_{[0,1]^{2}} p^{*}(y) dx dy -
\frac{1}{\lambda} \cdot
\int_{0 \leq y \leq x \leq 1} \int_{x}^{1} p^{*}(u) du dx dy \\
  & = &
\int_{0}^{1} p^{*}(y) dy - \frac{1}{\lambda} \cdot \int_{0}^{1} \frac{u^{2}}{2}
\cdot p^{*}(u) du \\
  & = &
\int_{0}^{1} \left( 1 - \frac{y^{2}}{2 \cdot \lambda} \right) \cdot
p^{*}(y) dy \\
  & = &
\int_{0}^{1} \left( 1 - \frac{y^{2}}{2 \cdot \lambda} \right) \cdot \left(
- 2 \cdot y \cdot \exp\left( \frac{y^{2}}{2 \cdot \lambda^{2}} \right) + 2
\cdot \lambda\right) dy \\
  & + &
\int_{0}^{1} \left( 1 - \frac{y^{2}}{2 \cdot \lambda} \right) \cdot \sqrt{2 \cdot \pi}
\cdot y \cdot \exp\left( \frac{y^{2}}{2 \cdot \lambda^{2}} \right) \cdot
\erf\left( \frac{y}{\sqrt{2} \cdot \lambda} \right) dy .
\end{eqnarray*}
The first integral is given by
$I_{1} = (\lambda - 2 \cdot \lambda^{2} - 2 \cdot \lambda^{3})
           \cdot
         \exp(1/(2 \cdot \lambda^{2}))
         - 1/3 
         + 2 \cdot (\lambda + \lambda^{2} + \lambda^{3})$.
The second integral we solve by integration by parts letting
$f^{\prime} = 
 \left( 1 - \frac{y^{2}}{2 \cdot \lambda} \right) 
   \cdot
  \sqrt{2 \cdot \pi} \cdot y \cdot \exp\left( \frac{y^{2}}{2 \cdot \lambda^{2}} \right)$
and
$g = \erf\left( \frac{y}{\sqrt{2} \cdot \lambda} \right)$.
Then $\int_{0}^{1}
f^{\prime} g dy = [f g]_{0}^{1} - \int_{0}^{1} f g^{\prime} dy$
is given by:
\begin{eqnarray*}
I_{2}
  & = &
\left[ \frac{\sqrt{\pi}}{\sqrt{2}} \cdot \lambda \cdot
(2 \cdot \lambda + 2 \cdot \lambda^{2} - y^{2} )
 \cdot
\exp(y^{2}/(2 \cdot \lambda^{2})) \cdot
\erf( y/(\sqrt{2} \cdot \lambda) ) \right] _{0}^{1} \\
  & - &
\int_{0}^{1} (2 \cdot \lambda + 2 \cdot \lambda^{2} - y^{2} ) dy .
\end{eqnarray*}
Combining all the terms in the sum $I_{1} + I_{2}$ we obtain
$2 \cdot \lambda^{3}$.
The third inner product is given by
\begin{eqnarray*}
\pair{\varphi}{\overline{\psi}}
  & = &
\int_{[0,1]^{2}} q(x) \cdot p^{*}(y) dx dy \\
  & - &
\frac{1}{\lambda} \cdot
\int_{0 \leq x \leq y \leq 1} \int_{x}^{y} q(t) dt \cdot p^{*}(y) dx dy
-
\frac{1}{\lambda} \cdot
\int_{0 \leq y \leq x \leq 1} q(x) \cdot \int_{x}^{1} p^{*}(u) du dx dy \\
  & = &
\left( \int_{0}^{1} q(x) dx \right)
 \cdot
\left( \int_{0}^{1} p^{*}(y) dy \right)
   -
\frac{2}{\lambda}
     \cdot
\int_{0 \leq t \leq y \leq 1} t \cdot q(t) \cdot p^{*}(y) dt dy \\
  & = &
2 \cdot \lambda^{3}
   -
2 \cdot \lambda
    \cdot
\int_{0}^{1}
    \left( 1 - \exp(-y^{2}/(2 \cdot \lambda^{2}) ) \right)
  \cdot
     p^{*}(y) dy \\
  & = &
- 2 \cdot \lambda^{3}
+ 2 \cdot \lambda \cdot
\int_{0}^{1} \exp(- y^{2}/(2 \cdot \lambda^{2}) ) \cdot p^{*}(y) dy \\
  & = &
- 2 \cdot \lambda^{3}
+ 4 \cdot \lambda \cdot \int_{0}^{1}
\left(
   - y
   + \exp( -y^{2}/(2 \cdot \lambda^{2}) )
   + \sqrt{\pi}/\sqrt{2} \cdot y \cdot
     \erf(y/(\sqrt{2} \cdot \lambda)) \right) dy \\
  & = &
2 \cdot \lambda^{2} \cdot \exp(- 1 / (2 \cdot \lambda^{2})) .
\end{eqnarray*}
\end{proof}

\subsection{Asymptotics}

We know that the largest root $\lambda_{0}$
of the eigenvalue
equation~(\ref{equation_213_eigenvalue})
is real, positive and simple,
since the associated operator $T$ is
positivity improving.
However, to say a bit more about the eigenvalues
consider the related equation
$\erf(z)=\sqrt{2/\pi}$.

Since the error function is an increasing function on the real axis, the
equation $\erf(z)=\sqrt{2/\pi}$ has a unique real root.
The error function is an odd function 
hence
we know by the strong version of the little
Picard theorem that the equation $\erf(z)=\sqrt{2/\pi}$ has
infinitely many roots.
Moreover, the complex roots appear in conjugate pairs.
To summarize this discussion we have:
The eigenvalue equation has a unique real root
which is positive and is the largest root.
The remaining infinitely many roots are all
complex and appear in conjugate pairs.

Numerically, we can approximate the roots
of the eigenvalue equation~(\ref{equation_213_eigenvalue}).
The unique real root is $\lambda_{0}=0.7839769312\ldots$.
The next four largest roots are:
\begin{eqnarray*}
\lambda_{1,2} & =0.2141426360\ldots\pm0.2085807022\ldots\cdot i \\
\lambda_{3,4} & =-0.1677323922\ldots\pm0.2418627350\ldots\cdot i
\end{eqnarray*}
Furthermore,
Proposition~\ref{proposition_213_simple}
states that all eigenvalues are simple.
Hence from Proposition~\ref{proposition_213_calculations}
we conclude:
\begin{theorem}
\label{theorem_213}
The number of $213$-avoiding permutations satisfies
$$
\frac{\alpha_{n}(213)}{n!}
   =
\exp\left(\frac{1}{2 \cdot \lambda_{0}^{2}}\right) \cdot \lambda_{0}^{n+1}
 +
O\left( | \lambda_{1} |^{n}\right)
$$
where $\lambda_{0}=0.7839769312\ldots$
is the unique real root of the equation
$\erf\left(1/(\sqrt{2} \cdot \lambda)\right)  =  \sqrt{{2}/{\pi}}$,
and $\lambda_{1}$ is the next largest root
and its modulus is given by
$|\lambda_{1}| = 0.298936411\ldots$.
\end{theorem}
By considering the next two conjugate roots
$\lambda_{1}$ and $\lambda_{2}$
we obtain an approximation for the next real term in the expansion of
$\alpha_{n}(213)/n!$:
$$ 2 \cdot 1.158597034\ldots \cdot (0.2989364111\ldots)^{n+1}
     \cdot 
           \cos\left(
- 5.593221320\ldots + (n+1) \cdot 0.7722415374\ldots
               \right) . $$

\section{123,231,312-Avoiding permutations}
\label{section_123_231_312}

Let us consider $123,231,312$-avoiding permutations.
Now the set $S$ is not a singleton, but has
cardinality three, that is, $S$ is given by
the set $\{123,231,312\}$.
The function $\chi$ is given by
$$
    \chi(x_{1},x_{2},x_{3})
  =
    \left\{ \begin{array}{c l}
         1  & \text{if } x_{3} \leq x_{2} \leq x_{1}, \\
         1  & \text{if } x_{2} \leq x_{1} \leq x_{3}, \\
         1  & \text{if } x_{1} \leq x_{3} \leq x_{2}, \\
         0  & \text{otherwise}.
     \end{array} \right.
$$
That is, the operator $T$ is described by
$$ (T f)(x,y)
    =
\left\{ \begin{array}{c l}
          \int_{x}^{y} f(t,x) dt &
          \text{ if }0 \leq x \leq y \leq 1, \\[2 mm]
          \int_{0}^{y} f(t,x) dt + \int_{x}^{1} f(t,x) dt &
          \text{ if } 0 \leq y \leq x \leq 1.
        \end{array} \right. $$

\subsection{Eigenvalues and eigenfunctions}

For a real number $t$ let 
$\{t\}$ denote the fractional part of the real number $t$,
that is, $\{t\} = t - \lfloor t \rfloor$. Observe that
the fractional part belongs to the interval $[0,1)$.
\begin{lemma}
The subspace $W$ of $L^{2}([0,1]^{2})$ consisting of
functions of the form 
\begin{equation}
f(x,y)  =  g(\{x-y\}) \:\:\:\: \text{ for } (x,y) \in [0,1]^{2}
\label{equation_123_231_312_f}
\end{equation}
is invariant under the operator $T$.
\end{lemma}
\begin{proof}
There are two cases to consider. First, if $x \leq y$, we have
$$
      T(g(\{x-y\}))
   =
      \int_{x}^{y} g(\{t-x\}) dt
   =
      \int_{0}^{y-x} g(s) ds  . $$
Second, if $y \leq x$ we have
\begin{eqnarray*}
      T(g(\{x-y\}))
  & = &
      \int_{0}^{y} g(\{t-x\} ) dt
     +
      \int_{x}^{1} g(\{t-x\}) dt \\
  & = &
      \int_{1-x}^{1+y-x} g(s) ds
     +
      \int_{0}^{1-x} g(s) ds
   =
      \int_{0}^{1+y-x} g(s) ds .
\end{eqnarray*}
These two cases can be summarized as the integral
from $0$ to $\{y-x\}$. However the upper bound
of the integral can be written as $1 - \{x-y\}$,
that is, the operator can be expressed as
$$      T(g(\{x-y\}))
    = 
       \int_{0}^{1-\{x-y\}} g(s) ds .  $$
\end{proof}
Since constant function $\mathbf{1}$ belongs to $W$,
it is enough to consider
eigenvalues and eigenfunctions in the subspace $W$.

\begin{proposition}
The eigenvalues of the operator $T$ restricted to
the invariant subspace $W$ are given by
$\lambda_{j} = (-1)^{(j-1)/2} \cdot 2/(\pi \cdot j)$
where $j$ is an odd positive integer
and the corresponding eigenfunctions are described by
$$
  \varphi(x,y) = \cos(\{x-y\}/\lambda) .
$$
\label{proposition_123_231_312_eigenfunction}
\end{proposition}
\begin{proof}
The defining relations for the eigenfunctions are
\begin{equation}
\lambda \cdot g(t)
     =
   \int_{0}^{1-t} g(s) ds  .
\label{equation_123_231_312_eigenvalue_condition}
\end{equation}
Observe that $\lambda = 0$ is not an eigenvalue.
Differentiating
equation~(\ref{equation_123_231_312_eigenvalue_condition})
twice, we obtain
\begin{eqnarray}
\lambda \cdot g^{\prime}(t)
  & = &
- g(1-t) ,
\label{equation_123_231_312_eigenvalue_condition_II} \\
\lambda \cdot g^{\prime\prime}(t)
  & = &
g^{\prime}(1-t)  = - 1/\lambda \cdot g(t) .
\label{equation_123_231_312_eigenvalue_condition_III}
\end{eqnarray}
This second order equation has the general solution
$A \cdot \cos(t/\lambda) + B \cdot \sin(t/\lambda)$.
Setting $t=1$
in~(\ref{equation_123_231_312_eigenvalue_condition})
implies that $g(1) = 0$ and then
setting $t=0$
in~(\ref{equation_123_231_312_eigenvalue_condition_II})
implies that $g^{\prime}(0) = 0$.
Hence $B=0$ and the eigenfunction is
$\cos(t/\lambda)$. 
For this function,
equation~(\ref{equation_123_231_312_eigenvalue_condition})
implies that
$\cos(t/\lambda) = \sin((1-t)/\lambda)$.
Setting $t=0$ we obtain that the eigenvalues
satisfy $\sin(1/\lambda) = 1$.
They can be parametrized as in the proposition.
\end{proof}

Note that the eigenvalues tend to $0$ as
$$  \lambda_{1} > -\lambda_{3} > \lambda_{5} > -\lambda_{7} > \cdots . $$

\begin{proposition}
The eigenvalues $\lambda$ of the operator $T$
restricted to the subspace $W$ are simple.
That is,
there is no function $h(t)$ such that
$\lambda \cdot h = T(h) + g$
where
$g(t) = \cos(t/\lambda)$.
\label{proposition_123_231_312_simple}
\end{proposition}
\begin{proof}
The defining relations for the eigenfunctions are
\begin{equation}
\lambda \cdot h(t)
    =  
\int_{0}^{1-t} h(s) ds + g(t) .
\label{equation_123_231_312_simple} 
\end{equation}
Differentiating this equation twice gives
\begin{eqnarray}
\lambda \cdot h^{\prime}(t)
  & = &
- h(1-t) + g^{\prime}(t) , 
\label{equation_123_231_312_simple_D}  \\
\lambda \cdot h^{\prime\prime}(t)
  & = &
h^{\prime}(1-t) + g^{\prime\prime}(t) .
\label{equation_123_231_312_simple_D_D}
\end{eqnarray}
Substituting~(\ref{equation_123_231_312_simple_D})
into~(\ref{equation_123_231_312_simple_D_D})
and using
(\ref{equation_123_231_312_eigenvalue_condition_III}),
we obtain the second order differential equation:
$$
    \lambda^{2} \cdot h^{\prime\prime}(t) + h(t)
  =
    2 \cdot \lambda \cdot g^{\prime\prime}(t)    . $$
The general solution is
$$
     h(t)
  =
     - \frac{1}{\lambda^{2}} \cdot t \cdot \sin(t/\lambda)
  +
     C_{1} \cdot \cos(t/\lambda)
  +
     C_{2} \cdot \sin(t/\lambda)  . 
$$
In the following we use that $\sin(1/\lambda) = 1$
and $\cos(1/\lambda) = 0$.
Observe that setting $t=1$ in~(\ref{equation_123_231_312_simple})
gives $\lambda \cdot h(1) = g(1)$, which implies $C_{2} = 1/\lambda^{2}$.
Similarly, setting $t=0$ in~(\ref{equation_123_231_312_simple_D})
gives $\lambda \cdot h^{\prime}(0) = -h(1) + g^{\prime}(0)$
which implies $C_{2} = 1/(2 \cdot \lambda^{2})$,
a different value for $C_{2}$. Hence there is no such a function $h$
satisfying the integral equation~(\ref{equation_123_231_312_simple}).
\end{proof}

Observe that the function $\chi$
satisfies the symmetry in Lemma~\ref{lemma_J}.
Hence the adjoint eigenfunction $\psi$ is given
by $J \varphi$. However, since the value of
$\varphi(x,y)$ only depends on the difference
$x-y$, we have that $J \varphi = \varphi$.

\begin{proposition}
For a non-zero eigenvalue $\lambda$ with eigenfunction $\varphi$
and eigenfunction $\psi = \varphi$ of the adjoint operator $T^{*}$,
we have
\begin{eqnarray*}
\pair{\varphi}{\mathbf{1}} & = & 
\pair{\mathbf{1}}{\overline{\psi}} = \lambda , \\
\pair{\varphi}{\overline{\psi}}    & = & 1/2 .
\end{eqnarray*}
In particular,
$$
\frac{\pair{\varphi}{\mathbf{1}} \cdot
      \pair{\mathbf{1}}{\overline{\psi}}}
     {\pair{\varphi}{\overline{\psi}}}
=
  2 \cdot \lambda^{2} .
$$
\label{proposition_123_231_312_calculations}
\end{proposition}

\subsection{Asymptotics}

The above computations show that all eigenvalues
of the operator $T$ are simple and
given explicitly.
We thus obtain the following 
asymptotic expansion for
$\pair{T^{n-2}(\mathbf{1})}{\mathbf{1}}
   =
\alpha_{n}(123,231,312)/n!$
as an immediate consequence of 
Theorem~\ref{theorem_expansion},
Propositions~\ref{proposition_123_231_312_eigenfunction}
and~\ref{proposition_123_231_312_calculations}:
$$
\frac{\alpha_{n}(123,231,312)}{n!}
   =
2 \cdot 
\sum_{\substack{ j = 1 \\ j \text{ \scriptsize odd}}}^{2k+1}
    \lambda_{j}^{n}
   +
 O\left( |\lambda_{2k+3}|^{n}\right)  .
$$
However, when we let $k$ tend to infinity,
observe that the right hand side converges.
Hence we obtain the exact expression
$$
\frac{\alpha_{n}(123,231,312)}{n!}
   =
2 \cdot 
\sum_{\substack{ j \geq 1 \\ j \text{ \scriptsize odd}}}
   (-1)^{(j-1)/2 \cdot n} \cdot \left(\frac{2}{\pi \cdot j}\right)^{n}
$$
for $n \geq 2$.
However this is the expression for 
$E_{n-1}/(n-1)!$ where $E_{n}$ denotes the $n$th Euler number;
see equation~(\ref{equation_Euler}).
Thus we conclude that
\begin{theorem}
\label{theorem_123_231_312}
The number of $123,231,312$-avoiding permutations
in ${\mathfrak S}_{n}$ is given by $n \cdot E_{n-1}$
for $n \geq 2$.
\end{theorem}
This result is due to Kitaev and Mansour~\cite{Kitaev_Mansour}.
Note that the special form of the eigenfunctions 
and the invariant subspace are reflected
in their proof that this class of permutations is invariant
under the shift
$\pi_{1} \pi_{2} \cdots \pi_{n} \longmapsto
 (\pi_{1}+1) (\pi_{2}+1) \cdots (\pi_{n}+1)$,
where the addition is modulo $n$.

\section{Concluding remarks}

It is straightforward to design a Viennot 
``pyramid'' to compute the number $\alpha_{n}$ of
$S$-avoiding permutations. For the original Viennot triangle,
see~\cite{Viennot_I,Viennot_II}. Let the entry
$\alpha_{n}^{i_{1},\ldots,i_{m}}$ of the pyramid be the number of
permutations in the symmetric group on $n$ elements, avoiding the set
$S$ and ending with the $m$ entries $i_{1},\ldots,i_{m}$. Then the
entry $\alpha_{n}^{i_{1},\ldots,i_{m}}$ is a sum of entries of the
form $\alpha_{n-1}^{j,i_{1},\ldots,i_{m-1}}$. This sum is a
discrete analogue of the operator $T$. How far does this analogue
between the discrete model and the continuous one go? Does the
function $f_{n}=T^{n-m}(\mathbf{1})$ approximate the $n$-th level of the
pyramid? More exactly, how well does the integer
$\alpha_{n}^{i_{1},\ldots,i_{m}}$ compare with $n!\cdot
f_{n}(i_{1}/n,\ldots,i_{m}/n)$?

In the case of descent pattern avoidance, can one prove that $T$ restricted
to the invariant subspace~$V$ is compact? We have done so in the case of
$123$-avoiding permutations.

Consider the graph $G_{\emptyset}$ of overlapping permutations
on the vertex set $\mathfrak{S}_{m}$.
What is the smallest number of edges
one has to remove in order to make the graph not strongly connected?
Clearly one can remove $m$ edges disconnecting
the vertex $12 \cdots m$.
Is $m$ the right answer? This would suggest that
one can remove $m-1$ edges without making the directed graph
disconnected.

A more general enumeration problem is as follows. For a function $w$ on
$\mathfrak{S}_{m+1}$ define the weight of a permutation
$\pi=\left(\pi_{1}\pi_{2}\cdots\pi_{n}\right)$ by the product
$$
\wt(\pi)  =  \prod_{k=1}^{n-m} w(\Pi(\pi_{k}, \ldots, \pi_{k+m})) .
$$
Now what can be said about the values and asymptotics of the sum
$$
\alpha_{n}(w)  =  \sum_{\pi \in \mathfrak{S}_{n}} \wt(\pi)
$$
as $n$ tends to infinity.
For instance, if $w$ is a positive function we
know by the result of Kre\u{\i}n and Rutman that
$$
\alpha_{n}(w)  \sim  c \cdot \lambda^{n} \cdot n! .
$$

An operator is called a {\em Volterra operator} if 
its spectral radius is $0$.
What can be said about Volterra operators of the form~(\ref{equation_T})?
More specifically, what can be said about
the asymptotic of
$\pair{T^{n}(\mathbf{1})}{\mathbf{1}}$?
See Examples~\ref{example_132_231},
\ref{example_123_213_21_321} and~\ref{example_312_321}
for different behaviors.

In all the cases we have calculated,
all the (non-zero) eigenvalues have been simple,
see Propositions~\ref{proposition_123_simple},
\ref{proposition_213_simple}
and~\ref{proposition_123_231_312_simple}.
Hence it is natural for us to conjecture:
\begin{conjecture}
When the graph $H_{S}$ is ergodic
then the associated operator $T_{S}$
only has simple eigenvalues, that is,
all the eigenvalues have index $1$.
\label{conjecture_simple}
\end{conjecture}
In the case when associated operator
satisfies a symmetry condition as
in Lemma~\ref{lemma_J}
or 
in Lemma~\ref{lemma_R},
does this condition help in proving the conjecture.
If Conjecture~\ref{conjecture_simple} is false
the asymptotic expansion can be obtained
by using
Theorem~22 in~\cite[Section~VII.3]{Dunford_Schwarz}.

In this paper our object is to understand consecutive patterns avoidance.
{\em Generalized pattern avoidance} (also known as {\em vincular pattern avoidance}, see Chapter 7 in \cite{kit}) was introduced by
Babson and Steingr\'{\i}msson~\cite{Babson_Steingrimsson}.
Is there an analytic approach to obtain
asymptotics for these classes of permutations?
Lastly, it would be daring to
ask for an analytic proof of the former Stanley--Wilf conjecture, recently
proved in~\cite{Marcus_Tardos}.

\section*{Acknowledgements}

We are grateful to Judith McDonald for a helpful
discussion regarding the Perron--Frobenius Theorem
and to Hua-Huai (Felix) Chern
who helped us to improve
Proposition~\ref{proposition_213_simple}.
The first author thanks MIT and Institute for Advanced Study
where this paper was completed.
The first author was partially supported by
National Science Foundation grants DMS-0200624
DMS-0902063, DMS-0835373 and CCF-0832797,
and
National Security Agency grant H98230-06-1-0072.

\newcommand{\journal}[6]{{\sc #1,} #2, {\em #3} {\bf #4} (#5), #6.}
\newcommand{\toappear}[3]{{\sc #1,} #2, to appear in {\em #3}}
\newcommand{\JCTA}{J.\ Combin.\ Theory Ser.\ A.}
\newcommand{\book}[4]{{\sc #1,} {\em #2}, #3, #4.}

\vanish{
\bibitem{Brislawn:1989}
\journal{C.\ Brislawn}
        {Kernels of trace class operators}
        {Proc.\ Amer.\ Math.\ Soc.}
        {104}{1988}{1181--1190}

\bibitem{Gantmacher:1960}
F.\ R.\ Gantmacher. \emph{The Theory of Matrices}.
Vols. 1, 2. Translated by K. A. Hirsch. Chelsea Publishing Co., New York 1959.

\bibitem{RS:1972}
M.\ Reed and B.\ Simon. \emph{Methods of Modern Mathematical
Physics, Volume I:\ Functional Analysis}. New York, London:\ Academic Press,
1972.

\bibitem{Simon:1979} B.\ Simon.
Trace ideals and their applications.
London Mathematical Society Lecture Note Series, {\bf 35}.
Cambridge University Press, Cambridge-New York, 1979.

\bibitem {Simon:2005} B.\ Simon.
Trace ideals and their applications. Second edition.
Mathematical Surveys and Monographs, 120.
American Mathematical Society, Providence, RI, 2005.

}

\bigskip

\noindent
{\em R.\ Ehrenborg, and P.\ Perry,}
Department of Mathematics,
University of Kentucky,
Lexington, KY 40506-0027, USA,
\{{\tt jrge},{\tt perry}\}{\tt @ms.uky.edu}

\bigskip

\noindent
{\em S.\ Kitaev,}
Department of Computer and Information Sciences,
University of Strathclyde,
Livingstone Tower, 26 Richmond Street,
Glasgow G1 1XH, United Kingdom,
{\tt sergey.kitaev@cis.strath.ac.uk}

\end{document}